\documentclass[11pt]{amsart}
\usepackage[margin=2.5cm, a4paper]{geometry}
\usepackage{amssymb, todonotes}

\title{
C*-envelopes of tensor algebras arising from stochastic matrices
}
\author{Adam Dor-On}
\address{Adam Dor-On, Pure Mathematics Department, University of Waterloo, Waterloo, ON, Canada N2L 3G1}
\email{adoron@uwaterloo.ca}

\author[Daniel Markiewicz]{Daniel Markiewicz}
\address{Daniel Markiewicz,
Department of Mathematics,
Ben-Gurion University of the Negev, 
P.O.B. 653, Be'er Sheva 8410501,
Israel.} 
\email{danielm@math.bgu.ac.il}
\date{\today}
\thanks{This work was partially supported by the ISF with the ISF-UGC joint research program framework (grant No. 1775/14). The work of the first author was also partially supported by an Ontario Trillium Scholarship. 
}

\newtheorem{theorem}{Theorem}[section]
\newtheorem{prop}[theorem]{Proposition}
\newtheorem{cor}[theorem]{Corollary}
\newtheorem{lemma}[theorem]{Lemma}
\newtheorem{remark}[theorem]{Remark}

\newtheorem{notation}[theorem]{Notation}

\theoremstyle{definition}
\newtheorem{defi}[theorem]{Definition}
\newtheorem{example}[theorem]{Example}
\numberwithin{equation}{section}

\newcommand{\nn}{\mathbb{N}}

\newcommand{\cc}{\mathbb{C}}
\newcommand{\torus}{\mathbb{T}}
\renewcommand{\epsilon}{\varepsilon}

\DeclareMathOperator{\Ker}{Ker}
\DeclareMathOperator{\Prod}{Prod}

\newcommand{\Ran}{\mathop{\mathrm{Ran}}\nolimits}

\newcommand{\ind}{\mathop{\mathrm{ind}}\nolimits}

\newcommand{\Rmnum}[1]{\expandafter\@\romancap\romannumeral #1@}

\newcommand{\la}{\langle}
\newcommand{\ra}{\rangle}

\newcommand{\St}{\Omega}

\newcommand{\tensor}{\mathcal{T}_+}
\newcommand{\toeplitz}{\mathcal{T}}
\newcommand{\cuntz}{\mathcal{O}}
\DeclareMathOperator{\Diag}{Diag}
\newcommand{\cala}{\mathcal{A}}
\newcommand{\calb}{\mathcal{B}}
\newcommand{\calc}{\mathcal{C}}

\newcommand{\cenv}{C^*_\mathrm{env}}
\newcommand{\LL}{\mathcal{L}}
\newcommand{\JJ}{\mathcal{J}}
\newcommand{\FF}{\mathcal{F}}
\newcommand{\KK}{\mathcal{K}}
\newcommand{\QQ}{\mathcal{Q}}
\newcommand{\Prim}{Prim}
\DeclareMathOperator{\Ext}{Ext}

\begin{document}

\begin{abstract}
In this paper we study the C*-envelope of the (non-self-adjoint) tensor algebra associated via subproduct systems to a finite irreducible stochastic matrix $P$.

Firstly, we identify the boundary representations  of the tensor algebra inside the Toeplitz algebra,
also known as its non-commutative Choquet boundary.  As an application, we provide examples of C*-envelopes that are not *-isomorphic to either the Toeplitz algebra or the Cuntz-Pimsner algebra.  This characterization required a new proof for the fact that the Cuntz-Pimsner algebra associated to $P$ is isomorphic to $C(\mathbb{T}, M_d(\cc))$, filling a gap in a previous paper.

We then proceed to classify the C*-envelopes of tensor algebras of stochastic matrices up to *-isomorphism and stable isomorphism, in terms of the underlying matrices. This is accomplished by determining the K-theory of these C*-algebras and by combining this information with results due to Paschke and Salinas in extension theory. This classification is applied to provide a clearer picture of the various C*-envelopes that can land between the Toeplitz and the Cuntz-Pimsner algebras.
\end{abstract}

\maketitle

\section*{Introduction}

Given a C*-correspondence $E$, the operator algebras associated to shift operators (also called creation operators) over the Fock correspondence $\mathcal{F}(E)$ have been the subject of considerable attention by too many researchers to appropriately list here.
By an \emph{operator algebra} in this paper we mean a (not necessarily self-adjoint) closed unital subalgebra $\cala$ of a unital $C^*$-algebra $\calb$. The operator algebra generated by the shifts in $\mathcal{L}(\mathcal{F}(E))$ is called the tensor algebra $\tensor(E)$, and it provides a very successful prototype for the study of operator algebras. It is closely related to the Toeplitz algebra $\toeplitz(E)$, which is the C*-algebra generated by the shifts, and its celebrated quotient, the Cuntz-Pimsner algebra $\cuntz(E)$.

Analogously, given a a subproduct system $X$ in the sense of Shalit and Solel~\cite{shalit-solel} of C*-correspondences over a C*-algebra $A$ parametrized by $\nn$, one obtains the operator algebras associated to shifts on $\mathcal{F}(X)$: the tensor algebra $\tensor(X)$, the Toeplitz algebra $\toeplitz(X)$ and the Cuntz-Pimsner algebra $\cuntz(X)$, where the latter was defined in \cite{viselter2}.
 This new framework generalizes the previous one, in the sense that a C*-correspondence $E$ gives rise to a \emph{product} system $X$ whose Fock correspondence and associated operator algebras are precisely the ones discussed in the previous paragraph. 

There has been important work on the operator algebras arising from subproduct systems over $\cc$, or equivalently, the special case of subproduct systems whose C*-correspondence fibers are actually \emph{Hilbert spaces}, see for example \cite{shalit-solel, davidson-ramsey-shalit, Kakariadis-Shalit}. In our previous paper \cite{dor-on-markiewicz}, we turned to the simplest case for which the fibers of the subproduct system are not Hilbert spaces. Namely, we considered the case of subproduct systems of C*-correspondences over $\ell^\infty(\Omega)$ when $\Omega$ is countable with more than one point.  Such a subproduct system and its associated operator algebras are conveniently parametrized by a stochastic matrix $P$ over the state space $\Omega$. In \cite{dor-on-markiewicz}, we considered isomorphism problems of the tensor algebras associated to stochastic matrices, via these subproduct systems. 

In this paper we focus on the \emph{C*-envelope} of the tensor algebra associated to finite irreducible stochastic matrices. Recall that given an operator algebra $\cala$, a C*-cover is a pair $(\calc,\iota)$ where $\calc$ is a C*-algebra and  $\iota: \cala \to \calc$ is a unital completely isometric homomorphism, such that $C^*(\iota(\cala)) = \calc$. A C*-cover is called the \emph{C*-envelope} for $\cala$ if for any other C*-cover $(\calc', \iota')$, the map $\iota'(a) \mapsto \iota(a)$ extends uniquely to a surjective *-homomorphism $\calc' \rightarrow \calc$. In this precise sense, the C*-envelope is the smallest C*-algebra which contains a completely isometric copy of $\cala$, and usually the algebra $\calc$ is denoted $\cenv(\cala)$ and
 the map $\iota$ is suppressed.

The existence of the C*-envelope of an operator algebra was first proven by Hamana~\cite{hamana-c-env}, by way of proving the existence of an injective envelope for operator systems. An alternative proof via dilation theory was found by Dritchel and McCullough in \cite{Dritschel-McCullough}. This new idea allowed Arveson~\cite{arveson-non-commutative-choquet-I} to follow
 the original strategy he envisioned in \cite{arv-subalg1, arv-subalg2}
 to prove the existence of the C*-envelope via boundary representations in the separable case. Davidson and Kennedy finally realized Arveson's vision in full in \cite{Davidson-Kennedy} by providing a proof without the assumption of separability.

We are motivated in this paper by the known results in the determination of the C*-envelope of the tensor algebra of a subproduct system: 

\begin{enumerate}
	\item
	 Given a C*-correspondence $E$, we have that 
	   $\cenv(\tensor(E)) = \cuntz(E)$. This was first proven by 
	 Muhly and Solel \cite[Corollary 6.6]{muhly-solel-tensor} when  the left action on $E$ is faithful, essential and acts by compacts, and in the general case (without extra assumptions) by
	 Katsoulis and Kribs \cite{katsoulis-kribs-C-envelope}.
	
	\item
	Let $I$ be a homogeneous ideal in the ring of polynomials in finitely many commuting variables. The universal commuting row contraction subject to the polynomial relations in $I$ gives rise to a subproduct system of Hilbert spaces $X^I$, and it was
	shown by Davdison, Ramsey and Shalit \cite{davidson-ramsey-shalit} that $\cenv(\tensor(X^I)) = \toeplitz(X^I)$.

	\item
	Let $I$ be a monomial ideal in the ring of polynomials in non-commuting variables. Similarly to the commutative case, a subproduct system $X^I$ can be defined. Kakariadis and Shalit \cite{Kakariadis-Shalit} have shown that for many cases (depending on the monomial ideal) either $\cenv(\tensor(X^I)) = \toeplitz(X^I)$  or $\cenv(\tensor(X^I)) = \cuntz(X^I)$. 
\end{enumerate}

In summary, for all these cases, when the subproduct system $X$ is irreducible in the appropriate sense (i.e. no nontrivial reducing projections, see Definition~\ref{rem:sps-reducing-proj}) , $\cenv(\tensor(X))$ has been found to be isomorphic either to $\toeplitz(X)$ or $\cuntz(X)$.  
In \cite[Section 6]{viselter2}, this phenomenon was observed, and it was asked if one can find a general description for the behavior of C*-envelopes of tensor algebras associated with subproduct systems. In this paper we shed some light on this question: we show that the evidence for the dichotomy witnessed above is misleading, and that the situation is more mysterious than previously thought, by providing an example of stochastic matrix with subproduct system X such that the 
$\cenv(\tensor(X))$ is not *-isomorphic to either $\toeplitz(X)$ or $\cuntz(X)$ (See Example \ref{ex:distinct}).

Our first main result is as follows. Let $P$ be a finite irreducible stochastic matrix. In this case we show that the C*-envelope
lands between the Toeplitz and Cuntz-Pimsner algebras in the sense that it fits in the following sequence of quotient maps:
\begin{equation}\tag{$\ast$}\label{quotient-sequence}
\toeplitz(P) \longrightarrow \cenv(\tensor(P)) \overset{\pi_P}{\longrightarrow} \cuntz(P)
\end{equation}
and in fact $\tensor(P)$ is hyperrigid inside $\cenv(\tensor(P))$ (See \cite{arveson-non-commutative-choquet-II}). Moreover, in the case when $P$ has the multiple arrival property (see Definition~\ref{def:multiple-arrival}),  we identify the boundary representations of $\tensor(P)$ inside $\toeplitz(P)$, also known as the non-commutative Choquet boundary in the sense of Arveson~\cite{arveson-non-commutative-choquet-I}. This enables us to describe the C*-envelope $\cenv(\tensor(P))$ in terms of boundary representations. For details see Corollary~\ref{cor:multiple-arrival-boundary} and Theorem~\ref{theorem:shilov-ideal-hyperrigid}.

The fact that the Cuntz-Pimsner algebra $\cuntz(X)$ as defined by Viselter~\cite{viselter2} is not always isomorphic to the C*-envelope of the tensor algebra $\tensor(X)$ in the subproduct system case, and even a dichotomy as above fails to hold, suggests that perhaps an alternative definition of Cuntz-Pimsner algebra for subproduct systems is desirable.

The concrete description of the C*-envelope and lack of dichotomy lead to an unexpected richness of possibilities. Our second main result concerns the classification of C*-envelope up to *-isomorphism and stable isomorphism, so as to clarify the situation. For a finite irreducible stochastic matrix $P$ over $\Omega^P$, the ideal $\Ker( \pi_P)$ in the sequence \eqref{quotient-sequence} is *-isomorphic to a direct sum of $n_P \leq |\Omega^P|$ copies of the algebra of compact operators. Given two irreducible stochastic matrices $P$ and $Q$ over finite state sets $\Omega^P$ and $\Omega^Q$ we have that

\begin{enumerate}
	\item
	$\tensor(P)$ and $\tensor(Q)$ have stably isomorphic C*-envelopes if and only if $n_P=n_Q$. For more details see Theorem \ref{thm:stable-iso}.
	
	\item
	$\tensor(P)$ and $\tensor(Q)$ have *-isomorphic C*-envelopes if and only if $|\Omega^P| = |\Omega^Q|$, $n_P = n_Q$ and up to a reordering of $\Omega^Q$, the matrices $P$ and $Q$ have the same column nullity in every column. For more details, see 
	Definition~\ref{def:column-nullity} and Theorem~\ref{thm:*-iso}.
\end{enumerate}
Therefore, we see that instead of a dichotomy, we actually have a profusion of possibilities.

These results are obtained by leveraging the surprisingly simple form
of the Cuntz-Pimsner algebra as obtained in \cite[Corollary 5.16]{dor-on-markiewicz} to compute the K-theory of $\cenv(\tensor(P))$.
However, in our case K-theory does not completely resolve the issue by itself, and we then use extension theory (especially work by Paschke and Salinas \cite{Paschke-Salinas}) to complete the task. We should note that 
Dilian Yang pointed out to us that there was a gap in the proof of 
\cite[Corollary 5.16]{dor-on-markiewicz}, which we resolve in Section \ref{Sec:Cuntz-Pimsner} of this paper.

Finally, its natural to ask about the relationship between 
$\cenv(\tensor(P))$ and the graph algebra $\cuntz_{Gr(P)}$ associated to the unweighted directed graph obtained from a finite irreducible stochastic matrix  $P$. We apply our classification results for the C*-envelope and K-theory for graph algebras to show that these two objects are generally incomparable in the sense that we exhibit $3 \times 3$ irreducible stochastic matrices $P$, $Q$ and $R$ such that
$$
\cenv(\tensor(P)) \not \sim \cenv(\tensor(Q)) \cong \cenv(\tensor(R)) \qquad \text{and} \qquad \cuntz_{Gr(P)} \cong \cuntz_{Gr(Q)} \not \sim \cuntz_{Gr(R)}
$$
where $\cong$ stands for *-isomorphism and $\sim$ stands for stable isomorphism.

This paper has six sections. In Section \ref{Sec:Prelim} we give some preliminary background required from \cite{dor-on-markiewicz} and on extension theory. In Section \ref{Sec:Cuntz-Pimsner} we fill the gap pointed out to us by Dilian Yang in the proof of \cite[Corollary 5.16]{dor-on-markiewicz} and compute the extension groups for the Cuntz-Pimsner algebra of a finite irreducible stochastic matrix. In Section \ref{Sec:Choquet} we determine the non-commutative Choquet boundary of $\tensor(P)$ inside $\toeplitz(P)$, which then allows us to compute C*-envelopes $\cenv(\tensor(P))$ associated to finite irreducible stochastic matrices. In Section \ref{Sec:K-theory} we compute the K-theory of $\cenv(\tensor(P))$ in terms of boundary representations. Finally, in Section \ref{Sec:Classification} we prove our main classification results mentioned above, and compare $\cenv(\tensor(P))$ and $\mathcal{O}_{Gr(P)}$ as invariants for the graph of $P$.

\section{Preliminaries} \label{Sec:Prelim}

\subsection*{Boundary representations and Shilov ideal}

Suppose that $\cala$ is a unital operator algebra, and $(\calb,\iota)$ is a C*-cover. A unital completely contractive (c.c.) map $\phi :\iota(\cala) \rightarrow B(\mathcal{H})$ has the \emph{unique extension property} if there exists a unique unital completely positive (c.p.) extension $\widetilde{\phi} : \calb \rightarrow B(H)$ which is also a *-representation. We will say that a unital *-representation $\rho : \calb \rightarrow B(\mathcal{H})$ is a \emph{boundary representation} for $\cala$ if $\rho$ is irreducible and $\rho \upharpoonright _{\iota(\cala)}$ has the unique extension property.

For a unital c.c. map $\phi : \cala \rightarrow B(\mathcal{H})$, we say that a unital c.c. map $\phi' : \cala \rightarrow B(\mathcal{H}')$ is a \emph{dilation} of $\phi$ if there is an isometry $V: \mathcal{H} \rightarrow \mathcal{H}'$ such that for any $a\in \cala$ we have $\phi(a) = V^*\phi'(a)V$. We will call a unital c.c. map $\phi : \iota(\cala) \rightarrow B(\mathcal{H})$ \emph{maximal} if whenever $\phi'$ is a dilation of $\phi$, then $\phi' = \phi \oplus \psi$ for some unital c.c. map $\psi$. It turns out that for a unital c.c. map $\phi : \iota(\cala) \rightarrow B(\mathcal{H})$ we have that $\phi$ is maximal if and only if $\phi$ has the unique extension property \cite[Proposition 2.2]{Arveson-note}, and that maximality is invariant under change of C*-cover \cite[Proposition 3.1]{Arveson-note}.

Thus, for the definitions of the unique extension property and maximality, it makes no difference which C*-cover $(\calb, \iota)$ we work inside.

Next, for a unital operator algebra $\cala$ and $(\calb,\iota)$ a C*-cover, an ideal $\mathcal{I}$ of $\calb$ is called a boundary ideal for $\cala$ if the canonical quotient map $q_{\mathcal{I}} : \calb \rightarrow \calb / \mathcal{I}$ is completely isometric on $\iota(\cala)$. The Shilov ideal $\mathcal{S}_{\cala}$ of $\cala$ in $\calb$ is the largest boundary ideal.

The Shilov ideal is a tractable tool for finding the C*-envelope, since $\calb / \mathcal{S}_{\cala}$ must then be the C*-envelope of $\cala$ (See \cite[Proposition 1.9]{kakariadis-sys-env}). However, there is a way to compute the C*-envelope from boundary representations. By the theorem of Davidson and Kennedy from \cite{Davidson-Kennedy}, the boundary representations of every unital operator systems completely norm it, so that by \cite[Theorem 2.2.3]{arv-subalg1} we have that the Shilov ideal is the intersection of all kernels of boundary representations.

In \cite{arveson-non-commutative-choquet-II} Arveson investigated a closely related notion for C*-covers called hyperrigidity. For a unital operator algebra $\cala$ and $(\calb,\iota)$ a C*-cover for it, one of the equivalent formulations for hyperrigidity of $(\calb,\iota)$ is that for \emph{any} $*$-representation $\pi : \calb \rightarrow B(H)$, the restriction $\pi |_{\iota(\cala)}$ has the unique extension property.

Suppose now that $\cala$ is hyperrigid in $(\calb, \iota)$. Then any irreducible *-representation of $\calb$ must be a boundary representation with respect to $\iota(A)$, so by taking the direct sum $\rho$ of all irreducible representations of $\calb$, by the above we have that the Shilov ideal of $\iota(\cala)$ in $\calb$ is trivial. This means that the C*-envelope of $\cala$ is $\calb$. Hence, when we know that a C*-cover is hyperrigid, this C*-cover must then be the C*-envelope for our operator algebra. By invariance of C*-envelope, we see that up to *-isomorphism, $\cala$ can only be hyperrigid in at most one C*-algebra.

We will often suppress the notation for the C*-cover that we use, and in many cases think of $\cala$ as a subalgebra of some particular $B(H)$.

\subsection*{Hilbert modules and subproduct systems}

We assume that the reader is familiar with the basic theory of Hilbert C*-modules, which can be found in \cite{Pasc, lance-book, Manu-troi-book}.
We only give a quick summary of basic notions and terminology as we go, so as to clarify our conventions.

Let $\cala$ be a C*-algebra and $E$ a Hilbert C*-module over $\cala$. We denote by $\mathcal{L}(E)$ the collection of adjointable operators on $E$. If in addition $E$ has a left $\cala$-module structure given by a *-homomorphism $\phi : \mathcal{A} \rightarrow \mathcal{L}(E)$, we call $E$ a Hilbert C*-correspondence over $\cala$. We often suppress notation and write $a \cdot \xi : =\phi(a)\xi$.

If $E$ is a C*-correspondence over $\cala$ with left action $\phi$, and $F$ is a C*-correspondence over $\cala$ with left action $\psi$, then on the algebraic tensor product $E\otimes_{alg}F$ one defines an $\cala$-valued pre-inner product satisfying $\la x_1 \otimes y_1 , x_2 \otimes y_2 \ra = \la y_1 , \psi(\la x_1 , x_2 \ra) y_2 \ra$ on simple tensors. The usual completion process with respect to the norm induced by this inner product, yields the internal Hilbert C*-module tensor product of $E$ and $F$, denoted by $E\otimes F$ or $E\otimes_{\psi}F$, which is a C*-correspondence over $\cala$ with left action $\phi \otimes Id_{F}$.

The following is the C*-algebraic version of \cite[Definition 1.1]{shalit-solel} for the semigroup $\mathbb{N}$, which was also given in \cite[Definition 1.4]{viselter1}.

\begin{defi}
Let $\cala$ be a C*-algebra, let $\{ X_n \}_{n\in \nn}$ be a family of Hilbert C*-correspondences over $\cala$ and let $U = \{ U_{n,m} : X_n \otimes X_m \rightarrow X_{n+m} \}$ be a family of bounded bimodule maps. We will say that $(X,U)$ is a subproduct system over $\cala$ if the following conditions are met:
\begin{enumerate}
\item
$X_0 = \cala$
\item
The maps $U_{0,n}$ and $U_{n,0}$ are given by the left and right actions of $\cala$ on $X_n$ respectively
\item
$U_{n,m}$ is an \emph{adjointable} coisometric map for every $n,m\in \nn$
\item
For every $n,m\in \nn$ we have the associativity identity
$$
U_{n+m,p}(U_{n,m}\otimes Id_{X_p}) = U_{n,m+p}(Id_{X_n}\otimes U_{m,p})
$$
\end{enumerate}
In case the maps $U_{n,m}$ are unitaries, we say that $X$ is a \emph{product system}.
\end{defi}

\subsection*{Operator algebras associated to subproduct systems}

We describe the construction of the tensor, Toeplitz and Cuntz-Pimsner algebras arising from subproduct systems (see \cite{viselter1, viselter2}). 

Let $(X,U)$ be a subproduct system over a C*-algebra $\cala$. There is a canonical \emph{product system} containing $(X,U)$ as a subproduct subsystem as follows (See \cite[Definition 5.1 \& Proposition 5.2]{shalit-solel}).

We define $E := X_1$, so that $\Prod(X) : = \{ E^{\otimes n} \}_{n\in \mathbb{N}}$ constitutes a product systems where the unitaries from $E^{\otimes n} \otimes E^{\otimes m}$ to $E^{\otimes n+m}$ are the usual associativity unitaries.

One can then construct canonical \emph{adjointable} coisometries $V_n : E^n \rightarrow X_n$ which, by associativity of $U = \{U_{n,m}\}$, are uniquely determined inductively by the equations $V_1 = Id_{X_1}$ and $V_{n+m} = U_{n,m} \circ (V_n \otimes V_m)$.

The $X$-Fock correspondence is the C* - correspondence direct sum of the fibers of the subproduct system

\begin{equation} \label{eq:fock-space}
\FF_X : = \bigoplus_{n\in \mathbb{N}}X_n
\end{equation}

Denote by $Q_n \in \LL(\FF_X)$ the projection of $\FF_X$ onto the $n$-th fiber $X_n$, and define $Q_{[0,n]} = Q_0 + ... + Q_n$, and $Q_{[n,\infty)} : = Id_{\FF_X} - Q_{[0,n-1]}$. We then obtain an adjointable coisometric map $V : \FF_{\Prod(X)} \rightarrow \FF_X$ given by $V = \oplus_{n=0}^{\infty}V_n$

The $X$-shifts are the operators $S_{\xi}^{(n)} \in \LL(\FF_X)$ uniquely determined between fibers by $S_{\xi}^{(n)}(\eta) : = U_{n,m}(\xi \otimes \eta)$ where $n,m \in \mathbb{N}$ and $\xi \in X_n$, $\eta \in X_m$.

We note that $S^{(n)}_{\xi} = V S^{(n)}_{V_n^*(\xi)} V^*$, so that $S_{\xi}^{(n)}$ is adjointable with adjoint given by $S^{(n)*}_{\xi} = V S^{(n)*}_{V_n^*(\xi)} V^*$, where $S^{(n)}_{V_n^*(\xi)}$ is a product system shift and is hence an adjointable operator in $\LL(\FF_{\Prod(X)})$.

\begin{defi}
The tensor and Toeplitz algebras are the non-self-adjoint and self-adjoint subalgebras of $\LL(\FF_X)$ generated by a copy of $\cala$ and all $X$-shifts respectively,
$$
\tensor(X) : = \overline{Alg}(\cala \cup \{S^{(n)}_{\xi} | \xi \in X_n , n\in \mathbb{N} \})
$$
$$
\toeplitz(X) : = C^*(\cala \cup \{S^{(n)}_{\xi} | \xi \in X_n , n\in \mathbb{N} \})
$$
\end{defi}

\begin{remark} \label{rem:fock-W*-vs-C*-sum}
When a subproduct system is comprised of $W^*$-correspondences, since each $S_{\xi}^{(n)}$ is adjointable, the last part of \cite[Proposition 2.14]{dor-on-markiewicz} allows us to take the $W^*$-correspondence (weak) direct sum of fibers as our Fock space in equation \eqref{eq:fock-space}, and get that the operator algebras $\toeplitz(X)$ and $\tensor(X)$ are the same as those considered in \cite[Definition 4.1 \& Definition 6.1]{dor-on-markiewicz}.
\end{remark}

The algebra $\LL(\FF_X)$ admits a natural action $\alpha$ of the unit circle $\mathbb{T}$ called the gauge action, defined by $\alpha_{\lambda}(T) = W_{\lambda} T W_{\lambda}^*$ for all $\lambda \in \mathbb{T}$ where $W_{\lambda} : \FF_X \rightarrow \FF_X$ is the unitary defined by 
$$
W_{\lambda}(\oplus_{n=0}^{\infty}\xi_n) = \oplus_{n=0}^{\infty}\lambda^n \xi_n
$$

Since $\alpha_{\lambda}(S_{\xi}^{(n)}) = S_{\lambda^n \xi}^{(n)}$ , we see that the algebras $\tensor(X)$ and $\toeplitz(X)$ are $\alpha$-invariant closed subalgebras, and so the action restricts to them, and we shall still denote it by $\alpha$.

The circle action on $\toeplitz(X)$ then enables the definition of a faithful conditional expectation $\Phi$ given by $\Phi(S) = \int_{\mathbb{T}}\alpha_{\lambda}(T)d\lambda$ where $d\lambda$ is normalized Haar measure on $\mathbb{T}$.

One then defines $\toeplitz(X)_k$ to be the closure of all homogeneous polynomials of degree $k$ (see \cite[Definition 4.5]{dor-on-markiewicz}), which then coincides with the collection of operators $T\in \toeplitz(X)$ satisfying $\alpha_{\lambda}(T) = \lambda^k T$ as shown in \cite[Corollary 4.6]{dor-on-markiewicz}. This makes both $\toeplitz(X)$ and $\tensor(X)$ into $\mathbb{Z}$-graded and $\mathbb{N}$-graded algebras respectively, and $\Phi$ on $\toeplitz(X)$ and $\tensor(X)$ is then onto $\toeplitz(X)_0$ and $\mathcal{A}$ respectively.

Another algebra associated to the subproduct system arises as a special quotient of $\toeplitz(X)$. The subset $\JJ \subset \LL(\FF_X)$ given by
$$
\JJ = \{ \ T \in \LL(\FF_X) \ | \ \lim_{n\rightarrow \infty} \| TQ_n \| = 0 \ \}
$$
is a closed $\alpha$-invariant left ideal inside $\LL(\FF_X)$ according to \cite[Proposition 4.8]{dor-on-markiewicz}. It was proven by Viselter in \cite[Theorem 2.5]{viselter2} that $\JJ(\toeplitz(X)): = \JJ \cap \toeplitz(X)$ is a closed two sided ideal.

\begin{defi}
Let $(X,U)$ be a subproduct system. Define the Cuntz-Pimsner ideal of $\toeplitz(X)$ to be $\JJ(\toeplitz(X)) : = \JJ \cap \toeplitz(X)$, and the Cuntz-Pimsner algebra of $X$ is then $\cuntz(X):= \toeplitz(X) / \JJ(\toeplitz(X))$.
\end{defi}

We note that the circle action on $\toeplitz(X)$ passes naturally to $\cuntz(X)$ since $\JJ(\toeplitz(X))$ is gauge invariant, and the fixed point algebras are then $\toeplitz(X)_0$ and $\cuntz(X)_0$ respectively.

We shall later need the following formula for the norm of an element in the quotient $M_s(\cuntz(X))$, in terms of representatives in $M_s(\toeplitz(X))$. Denote by $q: \toeplitz(X) \rightarrow \cuntz(X)$ the canonical quotient map. When $Q_n \in \toeplitz(X)$, it follows from item (1) of \cite[Theorem 3.1]{viselter2} that $\{I_s \cdot Q_{[0,m]}\}$ is an approximate identity for $M_s(\JJ(\mathcal{T}(X)))$, one may then invoke \cite[Exercise 1.8.C]{arv-book} to obtain the following.

\begin{prop} \label{prop:complete-norm-cuntz}
Let $(X,U)$ be a subproduct system, and suppose that $Q_n \in \toeplitz(X)$ for all $n\in \nn$. Then for any $T = [T_{ij}] \in M_s(\toeplitz(X))$ we have
$$
\|q^{(s)}(T)\| = \lim_{m \rightarrow \infty}\|[T_{ij}Q_{[m,\infty)}] \|
$$
\end{prop}

\subsection*{Cuntz-Pimsner algebras and subproduct systems arising from stochastic matrices}

In our previous paper~\cite{dor-on-markiewicz} we studied the tensor algebra $\tensor(P)$ associated to a certain subproduct system construction applied to a stochastic matrix $P$. This subproduct system construction can be applied to any unital normal completely positive map on a von-Neumann algebra, and is called the Arveson-Stinespring subproduct system construction.

After characterizing isomorphism classes for Arveson-Stinespring subproduct systems in terms of the underlying stochastic matrices, we used this characterization to study the dependence of the isomorphism classes of the algebra $\tensor(P)$ on the matrix $P$ (with respect to various concepts of isomorphism), which ended up coinciding with the respective isomorphism classes for the subproduct systems. 

We will now discuss some of the preliminaries and results in \cite{dor-on-markiewicz} for such subproduct systems and their Cuntz-Pimsner algebras. For the basic theory of stochastic matrices and Markov chains, we recommend \cite{seneta} and \cite[Chapter 6]{durrett}.

\begin{defi}
Let $\Omega$ be a countable set. A stochastic matrix is a function $P : \Omega \times \Omega \rightarrow \mathbb{R}_+$ such that for all $i\in \Omega$ we have $\sum_{j\in \Omega} P_{ij} = 1$. Elements of $\Omega$ are called \emph{states} of $P$.
\end{defi}

To every stochastic matrix, one can associate a set of edges $E(P) := \{ \ (i,j) \ | \ P_{ij}>0 \ \}$ and a $\{0,1\}$ - matrix $Gr(P)$ representing the directed graph of $P$ as an incidence matrix by way of 
\begin{displaymath}   
Gr(P)_{ij} = \left\{     
\begin{array}{lr}       
1 & : P_{ij} > 0 \\       
0 & : P_{ij} = 0     
\end{array}   \right.
\end{displaymath}

Many dynamical properties of $P$ can be put in terms of the directed graph $(\Omega, E(P))$ of $P$.

\begin{defi}
Let $P$ be a stochastic matrix over $\Omega$. A path of length $\ell$ in $(\Omega, E(P))$, i.e. a path in the directed graph of $P$, is a function $\gamma : \{0,..., \ell \} \rightarrow \Omega$ such that $P_{\gamma(k)\gamma(k+1)}>0$ for every $0\leq k \leq \ell -1$. The path $\gamma$ is said to be a cycle if $\gamma(0) = \gamma(\ell)$. We will say that a state $i$ leads to a state $j$ if there is a path $\gamma$ (of some length $\ell$) as above with $\gamma(0) = i$ and $\gamma(\ell)= j$.
\end{defi}

We next give the main definitions that we shall use in the context of stochastic matrices in this paper.

\begin{defi}
Let $P$ be a stochastic matrix over $\Omega$, and let $i\in \Omega$.
\begin{enumerate}
\item
The \emph{period} of $i$ is $r(i) = gcd\{ \ n \ | \ P_{ii}^{(n)} > 0 \ \}$. If no finite such $r(i)$ exists, or if $r(i) = 1$ we say that $i$ is aperiodic.
\item
$P$ is said to be irreducible if for any pair $i,j \in \Omega$, we have that $i$ leads to $j$ (and so $j$ also leads to $i$).
\end{enumerate}
\end{defi}

If $P$ is an irreducible stochastic matrix over $\Omega$, then every state $i\in \Omega$ is of the same periodicity $r$, so we define the periodicity of $P$ to be $r$.

Let us recall the statement of the cyclic decomposition of irreducible stochastic matrices \cite[Theorem 1.3]{seneta} which justifies the notion of periodicity of an irreducible stochastic matrix $P$.
 
\begin{theorem} \label{theorem:cyclic-graph-decomposition}
(Cyclic decomposition for periodic irreducible matrices)\\
Let $P$ be an irreducible stochastic matrix over a state set $\St$ with period $r$, and let $\omega \in\Omega$.  For each $\ell=0, \dots r-1$, let
$\Omega_\ell = \{ j \in \Omega \mid P_{\omega j}^{(n)}>0 \implies  n\equiv \ell 
\mod r \}$. Then, 
\begin{enumerate}
\item
The family $(\Omega_\ell)_{\ell=0}^{r-1}$ is a partition of 
$\Omega$. 
\item
If $j\in \Omega_{\ell}$ then there exists $N(j)$ such that for all $n\geq N(j)$ we have $P^{(nr + \ell)}_{\omega j}>0$. 
\item
Up to re-enumeration of $\Omega$, there exist \emph{rectangular} stochastic matrices $P_0,...P_{r-1}$ such that $P$ has the 
following cyclic block decomposition:
\begin{equation*}
\left[\begin{smallmatrix}
       0 & P_0  & \cdots & 0 \\
       \vdots  & \ddots  & \ddots & \vdots  \\
       0 &  \cdots & 0 & P_{r-2} \\
       P_{r-1}  & \cdots & 0 & 0
     \end{smallmatrix} \right]
\end{equation*}
where the rows (columns) of $P_{\ell}$ in this matrix decomposition are indexed by $\Omega_{\ell}$ ($\Omega_{\ell+1}$ respectively) for all for $\ell \in \mathbb{Z}_{r}$, where $\mathbb{Z}_r$ is the cyclic group of order $r$.
\end{enumerate}
\end{theorem}

In this paper we shall restrict our attention to \emph{finite irreducible} stochastic matrices. For this class of stochastic matrices, we have that the more generally stated \cite[Theorem 2.10]{dor-on-markiewicz}, yields the following cleaner formulation, which is a combination of \cite[Theorem 4.1]{seneta} and \cite[Theorem 6.7.2]{durrett}.

\begin{theorem} \label{theorem:convergence-theorem-finite}
(Convergence theorem for finite irreducible matrices) \\
Let $P$ be a \emph{finite irreducible} stochastic matrix over $\Omega$ with period $r\geq 1$, and $\St_0,...,\St_{r-1}$ a cyclic decomposition for it as in item (3) of Theorem~\ref{theorem:cyclic-graph-decomposition}. There exists a unique probability vector $\nu =(\nu_i)_{i\in \Omega}$ so that when we are given $i\in \St_{l_{1}}$ and $j\in \St_{l_{2}}$, for $0 \leq \ell < r$
such that $\ell \equiv (l_2-l_1) \mod r$, we have that
$$ \lim_{m\rightarrow \infty}P^{(mr+\ell)}_{ij} = \nu_jr.$$
\end{theorem}

Let $\Omega$ be a finite set and $\ell^{\infty}(\Omega) = C(\Omega) = \mathbb{C}^{\Omega}$ the C*-algebra of finite sequences indexed by $\Omega$. We denote by $\{ p_j \}_{j\in \Omega}$ the collection of pairwise perpendicular projections on $\ell^{\infty}(\Omega)$ given by $p_j(i) = \delta_{ij}$.

\begin{notation}
We denote by $*$ the Schur (entrywise) multiplication of matrices $A = [a_{ij}]$ and $B = [b_{kl}]$ given by $A * B = [a_{ij}b_{ij}]$, and let $\Diag$ be the map on matrices given by $\Diag([a_{ij}]) = (a_{ii})_{i \in \Omega} \in \ell^{\infty}(\Omega)$.

Next, for a non-negative matrix $P=[P_{ij}]$ indexed by $\Omega$, we denote by $\sqrt{P}$ and $P^{\flat}$ the matrices with $(i,j)$-th entry given by
$$
(\sqrt{P})_{ij} := \sqrt{P_{ij}}, 
\qquad \text{and} \qquad
(P^{\flat})_{ij} := 
\begin{cases} 
  (P_{ij})^{-1},  & \mbox{if } P_{ij} > 0 \\
  0, & \mbox{else} 
\end{cases}
$$
\end{notation}

In \cite[Theorem 3.4]{dor-on-markiewicz} the Arveson-Stinespring subproduct system associated to a stochastic matrix $P$ on countable $\Omega$ was computed. When $\Omega$ is finite, we arrive at the following simpler version of the theorem.

\begin{theorem}
Let $P$ be a stochastic matrix over \emph{finite} $\Omega$. The following is a subproduct system $Arv(P)$ over $C(\Omega) \cong \ell^{\infty}(\Omega)$ and is the one given in \cite[Theorem 3.4]{dor-on-markiewicz}.
\begin{enumerate}

\item
The $n$-th fiber is a $C^*$-correspondence over $C(\Omega)$ given by
$$
Arv(P)_n := \{ \ [a_{ij}] \ | \ a_{ij}= 0 \ \text{ if } \  (i,j) \notin E(P^n) \ \}
$$
with left and right actions of $C(\Omega)$ on $Arv(P)_n$ as a bimodule are given by diagonal left and right matrix multiplication and the
$C(\Omega)$-valued inner product is given by
$$
\langle A, B \rangle = \Diag \big[A^* B \big]
$$
for $A,B \in Arv(P)_n$.
\item
The subproduct maps are given by
$$
U_{n,m}(A \otimes B) = (\sqrt{P^{n+m}})^{\flat}*\big[(\sqrt{P^n} * A)\cdot(\sqrt{P^m} * B)\big] 
$$
for $n,m\neq 0$ and $A \in Arv(P)_n$ and $B \in Arv(P)_m$.

\end{enumerate}
\end{theorem}

\begin{remark}
Since the subproduct systems we shall consider in this work will be with finite dimensional fibers and over finite dimensional $C^*$-algebras, they will automatically be $W^*$-correspondences. Hence, by Remark \ref{rem:fock-W*-vs-C*-sum}, the theories of subproduct systems over $C^*$-algebras and their operator algebras discussed here and of subproduct systems over $W^*$-algebras and their operator algebras discussed in \cite{dor-on-markiewicz} will coincide. In this paper we choose to discuss our theories only in the C* (norm closed) context for the sake of brevity and a cleaner exposition.
\end{remark}

\begin{remark} \label{rem:sps-reducing-proj}
Let $(X,U)$ be a subproduct system of a C*-algebra $\cala$. A projection $p\in \cala$ is said to be reducing for $X$ if
$$
U_{n,m}^*(pX_{n+m}p) \subset pX_np \otimes pX_mp
$$

This is the C* / norm-closed version of \cite[Definition 6.19]{dor-on-markiewicz}. Using \cite[Proposition 7.4]{dor-on-markiewicz} we characterized the reducing projections of $Arv(P)$ for any stochastic matrix $P$ over $\Omega$. That is, there is a 1-1 bijection between reducing projections and subsets $C_p \subset \Omega$ such that whenever $\gamma : \{0,...,\ell\} \rightarrow \Omega$ is a path in the directed graph of $P$ that both begins and ends at $C_p$, then in fact for every $0 \leq k \leq \ell$ one has that $\gamma(k) \in C_p$.

Hence, for a finite irreducible stochastic matrix $P$ over $\Omega$, the only reducing non-zero projection is $1 = \sum_{i\in \Omega} p_i \in \ell^{\infty}(\Omega)$. Hence, by considering irreducible stochastic matrices, we are considering irreducible subproduct systems in the above dynamical sense.
\end{remark}

For $A\in Arv(P)_n$ we defined in \cite{dor-on-markiewicz} the shift operator $S^{(n)}_{A}$ uniquely determined on fibers by $S^{(n)}_A(B) = U_{n,m}(A\otimes B)$ for $B \in Arv(P)_m$. The Toeplitz and tensor algebras of $Arv(P)$ are given respectively by
$$
\toeplitz(P) = C^* \Big(\ell^{\infty}(\Omega) \cup \{ \ S^{(n)}_{A} \ | \ n\in \nn, \ A \in Arv(P)_n \ \} \Big)
$$
$$
\tensor(P) = \overline{Alg} \Big(\ell^{\infty}(\Omega) \cup \{ \ S^{(n)}_{A} \ | \ n\in \nn, \ A \in Arv(P)_n \ \} \Big)
$$
We note in passing that $\toeplitz(P)$ and $\tensor(P)$ are generated (as a C*-algebra, and as a norm-closed algebra respectively) by $\{p_i\}_{i\in \Omega}$ and $\{S_{E_{ij}}\}_{(i,j)\in E(P)}$, where $E_{ij} = [\delta_{ij}(k,l)]$ is the zero matrix, except for the $(i,j)$ entry at which it is $1$. Indeed, since $P$ is a finite matrix, each $S_{A}^{(n)}$ can be written as a finite linear combination of $S_{E_{ij}}^{(n)}$ with $(i,j)\in E(P^n)$. Then choose a path of length $n$, say $i = j_0 \rightarrow j_1 \rightarrow ... \rightarrow j_n = j$, and we have that $S_{E_{ij}}^{(n)} = c \cdot S_{E_{j_0j_1}}^{(1)} \cdot ... \cdot S_{E_{j_{n-1}j_n}}^{(1)}$ for some $c>0$.

Next, for a finite stochastic matrix $P$ over $\Omega$, and for every $n\in \nn$ and $A\in Arv(P)_n$ we defined operators in $\LL(\FF_{Arv(P)})$ mapping each $Arv(P)_m$ to $Arv(P)_{n+m}$, one denoted by $T_{A}^{(n)}$ and given by $T^{(n)}_{A} = S^{(n)}_{(\sqrt{P^{n}})^{\flat} * A}$, and the other denoted by $W_{A}^{(n)}$ which is uniquely determined on fibers $Arv(P)_m$ by $W^{(n)}_{A}(B) = A\cdot B$. For the purposes of computing the Cuntz-Pimsner algebra, we defined in \cite[Section 5]{dor-on-markiewicz} the auxiliary $C^*$-algebra
$$
\toeplitz^{\infty}(P) : = C^*\Big(\ell^{\infty}(\Omega) \cup \big\{ \ W^{(n)}_{A} \ | \ n\in \mathbb{N}, \  \ A \in Arv(P)_n \ \big\}\Big)
$$
and we noted that due to finiteness of $P$ we have that
$$
\toeplitz(P) = C^*\Big(\ell^{\infty}(\Omega) \cup \big\{ \ T^{(n)}_{A} \ | \ n\in \mathbb{N}, \  \ A \in Arv(P)_n \ \big\}\Big)
$$

\cite[Proposition 5.6]{dor-on-markiewicz} was then used to show that in fact $\cuntz(P)$ is *-isomorphic to $\toeplitz^{\infty}(P) / \JJ(\toeplitz^{\infty}(P))$, thereby reducing the computation of the Cuntz-Pimsner algebra to computing a quotient of an algebra generated by operators $W^{(n)}_{A}$ which do not depend on weights of entries of $P$.

\subsection*{Extension theory}

We recall some facts from the theory of primitive ideal spectra and extension theory for C*-algebras, to be used in later sections.

More details on primitive ideal spectra of C*-algebras can be found in \cite[Chapter 3]{Dix-C-alg} and \cite[Section 1.5]{arv-book}. For an account on the Busby invariant and extension theory for C*-algebras see \cite{Arveson-note-extension}, \cite[Section 15]{Blackadar-K-theory}, \cite[Section 1]{Brown-Dadarlat}, \cite[Section 2]{ELP-morphisms} and \cite{Paschke-Salinas}.

Let $A$ be a C*-algebra. We denote by $\hat{A}$ the collection of unitary equivalence classes of irreducible representations of $A$. On the other hand, we define $\Prim(A)$ to be the set of primitive ideals of $A$, where a primitive ideal is the kernel of an irreducible representation of $A$.

The set $\Prim(A)$ comes equipped with a lattice structure determined by set inclusion. Next, since any two unitarily equivalent *-representations have the same kernel, the map $\pi \mapsto \Ker \pi$ factors through to yield a surjective map $\kappa: \hat{A} \rightarrow \Prim(A)$.

It turns out that for type I C*-algebras, the above map $\kappa$ is a bijection, as every primitive ideal $J$ uniquely determines, up to unitary equivalence, an irreducible representation $\pi$ such that $J = \Ker \pi$ (See \cite[Theorem 4.3.7]{Dix-C-alg}).

When we have a *-isomorphism $\varphi : A \rightarrow B$ between two C*-algebras, we denote by $\varphi_* : \Prim(A) \rightarrow \Prim(B)$ the induced lattice isomorphism between the spectra. 

Suppose we have the following exact sequence of C*-algebras
\begin{equation} \label{exact-seq-one}
0 \rightarrow K \overset{\iota}{\rightarrow} A \overset{\pi}{\rightarrow} B \rightarrow 0
\end{equation}
and denote by $q : M(K) \rightarrow M(K) / K = : \QQ(K)$ the Calkin map. Then there is a *-homomorphism $\theta : A \rightarrow M(K)$ into the multiplier algebra of $K$, uniquely determined by $\theta(a) c = \iota^{-1} ( a \iota (c))$ for $c \in K$ and $a\in A$. Hence, a *-homomorphism $\eta : B \rightarrow \QQ(K)$ will be induced, and we call this map $\eta$ the \emph{Busby invariant} of the exact sequence above. We say that the above exact sequence is \emph{essential} if $K$ is an essential ideal in $A$, that is, if the intersection of $K$ with any non-trivial ideal in $A$ is non-trivial.

The above association turns out to be a bijection between exact sequences of C*-algebras given as in \eqref{exact-seq-one} and *-homomorphisms $\eta :B \rightarrow \QQ(K)$, where the inverse map sends a *-homomorphism $\eta :B \rightarrow \QQ(K)$ to the exact sequence where the pre-image $A := q^{-1}(\eta(B))$ under the Calkin quotient $q$, yield an exact sequence as in \eqref{exact-seq-one}, with $\pi$ replaced by the restriction of $q$ to $A$. Under this bijection, an exact sequence as in \eqref{exact-seq-one} is \emph{essential} if and only if its associated Busby invariant is an \emph{injective} *-homomorphism.

\begin{defi}
Suppose $K_i, A_i, B_i$ are C*-algebras for $i=1,2$, and that
\begin{equation} \label{eq:exact-seq}
0 \rightarrow K_1 \overset{\iota_1}{\rightarrow} A_1 \overset{\pi_1}{\rightarrow} B_1 \rightarrow 0
\ \ \text{and} \ \ 
0 \rightarrow K_2 \overset{\iota_2}{\rightarrow} A_2 \overset{\pi_2}{\rightarrow} B_2 \rightarrow 0
\end{equation}
are two short exact sequences. We say that these two short exact sequences are \emph{isomorphic} if there exists a *-isomorphism $\alpha : A_1 \rightarrow A_2$ such that $\alpha(\iota_1(K_1)) = \iota_2(K_2)$.
\end{defi}

Suppose $\eta_1$ and $\eta_2$ are Busby maps for exact sequences as in \eqref{eq:exact-seq}. \cite[Theorem 2.2]{ELP-morphisms} then yields that these two short exact sequences are isomorphic if and only if there exist *-isomorphisms $\kappa : K_1 \rightarrow K_2$ and $\beta : B_1 \rightarrow B_2$ such that 
$$
\widetilde{\kappa} \eta_1 = \eta_2 \beta
$$
where $\widetilde{\kappa} : \QQ(K_1) \rightarrow \QQ(K_2)$ is the induced *-isomorphism between the Calkin algebras.

In the context of extensions by a single copy of compact operators on separable infinite dimensional Hilbert space, that is when $K = \KK(H)$, the Calkin quotient map $q: M(\KK(H)) \rightarrow \QQ(\KK(H))$ discussed above is just the regular quotient map into the Calkin algebra, since $M(\KK(H)) = B(H)$, so that $M(\KK(H)) / \KK(H) = \QQ(H)$. We denote by $\KK = \KK(H)$ the compact operators on separable infinite dimensional Hilbert space $H$.

Let $B$ be a C*-algebra. We write $E(B)$ for the collection of all injective *-homomorp\-hisms of $B$ into $Q(H)$. We call elements in $E(B)$ extensions, as they are in bijection, under (the inverse of) the Busby map, with essential exact sequences of $C^*$-algebras of the form
$$
0 \rightarrow \KK \rightarrow A \rightarrow B \rightarrow 0
$$

We then say that two extensions $\eta_1, \eta_2 \in E(B)$ are 
\begin{enumerate}
\item
Strongly (unitarily) equivalent if there is a unitary $U\in B(H)$ such that $\eta_1(b) = q(U)\eta_2(b)q(U^*)$ for all $b\in B$.
\item
Weakly (unitarily) equivalent if there is a unitary element $u\in \QQ(H)$ such that $\eta_1(b) = u\eta_2(b)u^*$ for all $b\in B$.
\end{enumerate}

When $B$ is unital we write $\Ext_s(B)$ and $\Ext_w(B)$ for the strong and weak equivalence classes of \emph{unital} extensions in $E(B)$, respectively. When $B$ is non-unital, we write $\Ext_s(B)$ and $\Ext_w(B)$ for the strong and weak equivalence classes of \emph{all} extensions in $E(B)$, respectively, however in this case $\Ext_s(B) = \Ext_w(B)$ by \cite[Proposition 15.6.4]{Blackadar-K-theory}. We denote by
 $[\eta]_s$ and $[\eta]_w$ the equivalence classes of an extension $\eta$ in $\Ext_s(B)$ and $\Ext_w(B)$, respectively

 Given $\eta_1,\eta_2 \in E(B)$, we may define $\eta_1 \oplus \eta_2 \in E(A)$ (via some fixed identification $\QQ(H) \oplus \QQ(H) \subseteq \QQ(H \otimes \mathbb{C}^2) \cong \QQ(H)$) by specifying $(\eta_1 \oplus \eta_2)(b) = \eta_1(b) \oplus \eta_2(b)$. This operation induces a well-defined addition $+$ on $\Ext_s(B)$ and $\Ext_w(B)$ given for two extensions $\eta_1$ and $\eta_2$ by $[\eta_1]_s + [\eta_2]_s: = [\eta_1 \oplus \eta_2]_s$ and $[\eta_1]_w + [\eta_2]_w: = [\eta_1 \oplus \eta_2]_w$, and makes them into abelian semigroups.

An extension $\tau$ is called \emph{trivial}  if it lifts to a *-homomorphism $\hat{\tau} : B \rightarrow B(H)$ such that its composition with the Calkin quotient map yields $q \circ \hat{\tau} = \tau$. Such a trivial extension $\tau$ is called \emph{strongly unital}
if the map $\hat{\tau}$ can be chosen to be unital (in particular this
is relevant only when $B$ itself is unital and $\tau$ is unital). Trivial extensions correspond to \emph{split essential} exact sequences via (the inverse of) the Busby map. It is straightforward to construct injective *-homomorphisms of a C*-algebra $B$ into $B(H)$ which do not intersect $\KK(H)$, hence trivial extensions always exist. Moreover, the same argument yields strongly unital trivial extensions.

Voiculescu~\cite{Voicu-Weyl-VN} showed that when $B$ is separable, the semigroup $\Ext_s(B)$ has a zero element. When $B$ is non-unital, the zero element consists precisely of the trivial extensions. When $B$ is
unital, it consists of the strongly unital trivial extensions. For more
details, see \cite[Section~15.12]{Blackadar-K-theory}, especially \cite[Theorem~15.12.3]{Blackadar-K-theory}.

Although $\Ext_s(B)$ and $\Ext_w(B)$ are not always groups, it follows from a theorem of Choi and Effros that when $B$ is separable and nuclear,
both semigroups are actually groups (see \cite[Corollary 15.8.4]{Blackadar-K-theory}).
 
Suppose now that $B$ is unital. There is an action $\epsilon$ of $\mathbb{Z}$ on $\Ext_s(B)$ given by $\epsilon(m) [\eta]_s = [Ad_u\circ \eta]_s$ where $u\in \QQ(H)$ is a unitary of Fredholm index $-m$, and $Ad_u(a) = u^*au$ for $a\in \QQ(H)$. By definition of addition,
we have that $\epsilon(n+m)([\eta_1]_s + [\eta_2]_s) =
[Ad_{u \oplus v} (\eta_1 \oplus \eta_2) ] 
= \epsilon(n)[\eta_1]_s + \epsilon(m)[\eta_2]_s$ where $u$ and $v$
are unitaries in $\QQ(H)$ of indices $-n$ and $-m$ respectively. 
In particular, if $\tau$ is a strongly unital trivial extension
then $\epsilon(m)[\eta]_s = \epsilon(0+m)([\eta]_s + [\tau]_s)=
[\eta]_s + \epsilon(m)[\tau]_s$. Hence, when we denote by $\lambda_B: \Ext_s(B) \to \Ext_w(B)$ the canonical quotient map, we have that $\Ker \lambda_B = \{\epsilon(m)[\tau]_s \,|\, m \in \mathbb{Z} \}$. 

Let $\gamma_B : \Ext_w(B) \rightarrow \mathrm{Hom}(K_1(B),\mathbb{Z})$ be the so-called \emph{index invariant} of $B$, given by $\gamma_B([\eta]_w) = \ind \circ \eta_*$, where $\eta_* : K_1(B) \rightarrow K_1(\QQ(H))$ is the map induced between the $K_1$ groups and $\ind : K_1(\QQ(H)) \rightarrow \mathbb{Z}$ is the Fredholm index.
Hence, for a unital C*- algebra $B$, we always have the following sequence of maps
\begin{equation}\label{eq:seq-ext}
\Ext_s(B) \overset{\lambda_B}{\longrightarrow} \Ext_w(B) \overset{\gamma_B}{\longrightarrow} \mathrm{Hom}(K_1(B),\mathbb{Z})
\end{equation}

We next give the details of two particular examples, which will turn out to be useful to us later in the end of Section \ref{Sec:Cuntz-Pimsner} and in Section \ref{Sec:Classification}.

\begin{example} \label{ex:torus-extensions}
Take $B = C(\mathbb{T})$. In this case $B$ is nuclear and separable, so both the weak and strong extension semigroups are groups. We note that $\mathrm{Hom}(K_1(B),\mathbb{Z}) \cong \mathbb{Z}$ as $K_1(B) \cong \mathbb{Z}$, and every homomorphism is determined on the generator $1$. We next show that in this case, the map $\gamma_{B} \circ \lambda_B$ is surjective. Indeed, for every $m \in \mathbb{Z}$ there is a unitary $u \in \QQ(H)$ with $\sigma(u) = \mathbb{T}$, and Fredholm index $m$, we may define a *-homomorphism $\eta_m : C(\mathbb{T}) \rightarrow \QQ(H)$ given by $\eta_m(z\mapsto z) = u$ which implements
a *-isomorphism $C(\mathbb{T}) \cong C^*(u)$. Thus we obtain an extension with index invariant $k \mapsto k \cdot m \in \mathrm{Hom}(K_1(B),\mathbb{Z})$.

Next, we show that $\gamma_B \circ \lambda_B$ is injective. Indeed, if $\gamma_{B} \circ \lambda_B[\eta]_s = 0$, then $\ind(\eta(z \mapsto z)) = 0$ and hence there is a unitary $U \in B(H)$ with $\sigma(U) = \mathbb{T}$ s.t $q(U) = \eta(z \mapsto z)$. Thus, $\eta$ lifts to a  unital $*$-homomorphism $\hat{\eta} : C(\mathbb{T}) \rightarrow B(H)$, so that $\eta$ is a strongly unital trivial extension, and the map $\gamma_B \circ \lambda_B$ is injective.

We conclude that $\Ext_s(C(\mathbb{T})) \cong \Ext_w(C(\mathbb{T})) \cong \mathbb{Z}$, and that $\epsilon(n)$ acts trivially on $\Ext_s(C(\mathbb{T}))$ for each $n$.
\end{example}

\begin{example} \label{ex:matrix-extension}
Take $B = M_d$. Again in this case $B$ is nuclear and separable so that both weak and strong extension semigroups are groups. We already know that $K_1(M_d) \cong \{0\}$, so that the right most group in \eqref{eq:seq-ext} vanishes. Let $\eta : M_d \rightarrow \QQ(H)$ be
a unital extension. We reiterate the construction in \cite[Example 15.4.1 (b)]{Blackadar-K-theory} lifting $\eta : M_d \rightarrow Q(H)$ to a *-homomorphism $\hat{\eta} : M_d \rightarrow B(H)$, and measuring how far $\hat{\eta}$ is from being unital. That is, how far is $\eta$ from being a strongly unital trivial extension.

Let $\{\overline{e_{ij}}\}$ be a system of matrix units for $\eta(M_d)$. By standard essential spectrum arguments, one can find projections $p_{ii} \in B(H)$ that lift each $\overline{e_{ii}}$. Next, by appealing to \cite[Lemma~1.1]{Paschke-Salinas}, for all $2\leq i \leq d$  we may find partial isometries $e_{1i}$ lifting $\overline{e_{1i}}$ such that $e_{1i}^* e_{1i} \leq p_{ii}$ and $e_{1i}e_{1i}^* \leq p_{11}$. We set $e_{ij} = e_{1i}^* e_{1j}$ so that $\{ e_{ij}\}$ is a lifted set of matrix units in $pB(H)p$, where $p = \sum e_{ii}$. We note that $p$ is a projection of finite dimensional cokernel, say of dimension $\ell$, so that by adding a homomorphism from $M_d$ to $(1-p)B(H)(1-p)$ if necessary, we may arrange for $0 \leq \ell < d$. 

The defect of $\eta$ is then defined to be $\ell \in \mathbb{Z}_d$, and up to strong equivalence it is independent of the choice made in the process above. It is then easy to see that two unital extensions $\eta_1, \eta_2 \in E(M_d)$ are strongly equivalent if and only if they have the same defect, and are always weakly equivalent.
Hence, we conclude that $\Ext_s(M_d) \cong \mathbb{Z}_d$ and $\Ext_w(M_d) \cong \{0\}$.
\end{example}

\section{Cuntz-Pimsner algebra of a stochastic matrix} \label{Sec:Cuntz-Pimsner}

We next close a gap kindly pointed out to us by Dilian Yang in the proof of the characterization of the Cuntz-Pimsner algebra of a finite irreducible stochastic matrix, which is one of the theorems of \cite[Section~5]{dor-on-markiewicz}. The theorem at stake, which corresponds to \cite[Corollary 5.16]{dor-on-markiewicz} is
as follows.

\begin{theorem}\label{theorem:main}
Let $P$ be an irreducible stochastic matrix of size $d$. Then
$\cuntz(P) \cong M_d(\cc) \otimes C(\torus)$.
\end{theorem}

The main issue is that the cyclic decomposition of periodic irreducible stochastic matrices need not be realized in square blocks
as claimed in \cite[Remark~2.9]{dor-on-markiewicz}. Consider the following example
kindly brought to our attention by Dilian Yang: let $\Omega = \{1,2,3\}$ and set 
$$
P =  \begin{bmatrix}
       0 & 0  & 1 \\
       0  & 0  & 1 \\
       \frac{1}{2} & \frac{1}{2} & 0 
     \end{bmatrix} = 
     \begin{bmatrix}
            0  & P_0 \\
            P_1 & 0 
          \end{bmatrix}, \qquad
          \text{where} \quad
          P_0 = \begin{bmatrix}
                  1 \\
                  1 
               \end{bmatrix}, 
               \quad
               P_1=
           \begin{bmatrix}
                  \frac{1}{2} & \frac{1}{2} 
                \end{bmatrix}    
$$ 
The matrix $P$ has period $2$, $\Omega_0 = \{1,2\}$ and $\Omega_1 = \{3\}$, and both $P_0$ and $P_1$ are not square. 

Therefore, the results in 
\cite[Section 5]{dor-on-markiewicz}
after \cite[Proposition 5.7]{dor-on-markiewicz}, and in particular
\cite[Proposition 5.15]{dor-on-markiewicz}, only apply in the case
of square blocks. Therefore a gap remains in the proof of Theorem~\ref{theorem:main} in its stated form, and we now provide a different proof for it, which works in all cases, and resolves any remaining gaps with the rest of \cite[Section 5]{dor-on-markiewicz}. The issue above does not affect the remainder of the paper, namely \cite[Sections 6 and 7]{dor-on-markiewicz}.

Recall that the adjoint
$W_A^{(n)*}$ of $W_A^{(n)}$, which maps $Arv(P)_{n+m}$ to $Arv(P)_m$, is uniquely determined on fibers by
$$
W_A^{(n)*}(B) = Gr(P^{m}) * ( A ^* B ), \qquad  B \in Arv(P)_{m+n}
$$
where the reason for Schur-multipling $A ^* \cdot B$ with $Gr(P^{m})$ is to make sure that the product lands in $Arv(P)_m$ with its given entry constraints (See the discussion preceeding \cite[Proposition 5.7]{dor-on-markiewicz}).

We note that $\ell^\infty(\Omega)$ acts on $Arv(P)_m$ as left multiplication by diagonal matrices. Therefore,
$W_{E_{kk}}^{(0)} = p_k$ as the adjointable operator on $\mathcal{F}_{Arv(P)}$.

The following proposition, which works in all cases, replaces \cite[Remark 5.10]{dor-on-markiewicz} and the discussion preceding it.

\begin{prop} \label{proposition:nograph}
Let $q \in \nn$ and suppose that $A \in Arv(P)_q$. Then there exists
$m_0 \in \nn$ such that for all $m\geq m_0$ we have that
$$
W^{(q)*}_{A}(B) = A^*B, \qquad \forall B \in Arv(P)_{q+m}
$$
That is, if $m\geq m_0$ and $B \in Arv(P)_{q+m}$, then the matrix $A^*B$ has support contained in the support of $P^m$.
\end{prop}

\begin{proof}
Suppose this fails. Then there exists a sequence of matrices $B(n) \in Arv(P)_{q+m_n}$, with $n\mapsto m_n$ increasing, such that
the support of $A^*B(n)$ is not contained in the support of $P^{m_n}$.
By finiteness of $P$, perhaps by replacing $B(n)$ by a subsequence, we may assume that
there exist $i,j, k \in \Omega$ independent of $n$ such that
 $B(n)\in Arv(P)_{q+m_n}$ and both $p_iA^*p_k \neq 0$ and $p_k B(n)p_j \neq 0$ while $P_{ij}^{(m_n)} = 0$. Again by moving to a subsequence, we may assume that there exists $0\leq \ell < r$ independent of $n$
such that $m_n \equiv \ell \mod r$. 

Let $\Omega_0, \dots, \Omega_{r-1}$ be the cyclic decomposition of $P$ with respect to $k$. Note that by item (2) of Theorem \ref{theorem:cyclic-graph-decomposition}, we must have that $P^{(m)}_{ij} = 0$ for \emph{all} $m$ such
that $m\equiv \ell \mod r$. Therefore there are no paths from $i$
to $j$ whose length is of residue $\ell \mod r$.

Let  $\sigma(i), \sigma(j)$ be such that $i \in \Omega_{\sigma(i)}$ and
$j \in \Omega_{\sigma(j)}$.  Since $p_iA^*p_k \neq 0$, we have that $p_k A p_i \neq 0$ and hence $P^{(q)}_{k i}>0$. It follows from the cyclic decomposition theorem that paths from $k$ to $i$ have length with residue $\sigma(i)$ (mod $r$, which we will suppress). Since paths from $k$ to $k$ must have lengths with zero residue by periodicity, we must have that 
 paths from $i$ to $k$ will have length with residue $r-\sigma(i)$. Therefore paths from $i$ to $j$ have lengths with residue $r - \sigma(i) + \sigma(j) \equiv \sigma(j) - \sigma(i) \mod r$.

Next, since $A \in Arv(P)_q$, 
and $p_k A p_i \neq 0$, we have by the definition of the cyclic
decomposition that $\sigma(i) \equiv q \mod r$. Similarly, since
$B \in Arv(P)_{q + m_n}$, 
and $p_k B p_j \neq 0$, we have that $\sigma(j) \equiv q  + m_n  \equiv q + \ell \mod r$. Therefore, $\sigma(j) - \sigma(i) \equiv  q + \ell - q \equiv \ell \mod r$ and we conclude that all paths from $i$ to $j$ must have residue $\ell$ mod $r$. But this is impossible since we have noted before that there are \emph{no} paths from $i$ to $j$ whose length is of residue $\ell \mod r$.
\end{proof}

\begin{defi} \label{defi:cuntz-generators}
Let $P$ be a finite irreducible $r$-periodic stochastic matrix over $\Omega$ of size $d$.
We will say that a cyclic decomposition $\St_0,..., \St_{r-1}$ for $P$
is \emph{properly enumerated} if $\Omega$ is enumerated in such a way that for every $0\leq  m<k<r$,
$i \in \Omega_m$ and $j \in \Omega_k$ we have that $i<j$.
For $i \in \Omega$, denote by $\sigma(i)$ the unique index $0 \leq \sigma(i) < r$ such that $i \in \Omega_{\sigma(i)}$.

Given a properly enumerated cyclic decomposition $\St_0,..., \St_{r-1}$ for $P$, we define operators $U$ and $(S_{ij})_{i, j \in \Omega}$  in
$\LL(\mathcal{F}_{Arv(P)})$ as follows. The operator $U$ has degree
$r$ with respect to the grading, i.e. for every $m\in\nn$, $U (Arv(P)_m ) \subseteq Arv(P)_{m+r}$, and it is uniquely determined by 
$$
U(B) = Gr(P^{m+r}) * B, \qquad m \in \nn, B \in Arv(P)_m.
$$
If $i \leq j$, then $\sigma(i) \leq \sigma(j)$, and denote by $\ell = \sigma(j) - \sigma(i)$. Then $S_{ij}$ is an operator of degree $\ell$, i.e.
for all $m\geq 0$, $S_{ij}( Arv(P)_m ) \subseteq Arv(P)_{m+\ell}$ and 
it is given by
$$
S_{ij}(B) = Gr(P^{m+\ell})*(E_{ij} \cdot B), \qquad m\in\nn, B \in  Arv(P)_m.
$$
If $i > j$ we define $S_{ij} = S_{ji}^*$.
The family $(U, (S_{ij})_{i,j\in \Omega})$ is called the \emph{standard family} associated to the properly enumerated cyclic decomposition
$\St_0,..., \St_{r-1}$.
\end{defi}

Recall the following auxiliary C*-algebra considered in \cite[Section 5]{dor-on-markiewicz}:
\begin{align*}
\toeplitz^\infty(P) & = C^*( \ell^{\infty}(\Omega) \cup \{ W^{(n)}_A \mid n\in \nn, A \in Arv(P)_n \})
\end{align*}
We recall that
$$
\mathcal{J}(\mathcal{T}^{\infty}(P)) : = \{ \ T\in \mathcal{T}^{\infty}(P) \ | \ \lim_{n \rightarrow \infty}\|TQ_n \| = 0 \ \}
$$
is a two sided ideal in $\mathcal{T}^{\infty}(P)$ with $\mathcal{O}(P) \cong \mathcal{T}^{\infty}(P) / \mathcal{J}(\mathcal{T}^{\infty}(P))$ by \cite[Theorem 5.6]{dor-on-markiewicz}.
We denote by $\overline{T} \in \mathcal{O}(P)$ the image of $T\in \mathcal{T}^{\infty}(P)$ under the associated canonical quotient map $q: \mathcal{T}^{\infty}(P) \rightarrow \mathcal{O}(P) \cong  \mathcal{T}^{\infty}(P) / \mathcal{J}(\mathcal{T}^{\infty}(P))$.

\begin{lemma} \label{lemma:almost-universal}
Let $P$ be an irreducible $r$-periodic stochastic matrix over $\Omega$ of size $d$ with
properly enumerated cyclic decomposition $\St_0,..., \St_{r-1}$,
and let $(U, (S_{ij})_{i,j\in \Omega})$ be its associated 
standard family.
\begin{enumerate}
\item\label{relations-to-cuntz} Let $i,j\in \Omega$ be such that $i \leq j$ in the properly enumerated decomposition of $\Omega$, so that $\ell := \sigma(j) - \sigma(i) \geq 0$.
Then there exists $n_0 \in \mathbb{N}$ such that for all $n\geq n_0$ we have
$$
S_{ij} = W_{I_d}^{(nr)*}W_{E_{ij}}^{(nr+\ell)} \ \text{ and } \ U = W_{I_d}^{(nr)*}W_{I_d}^{(nr +r)}
$$
Hence, $U \in \toeplitz^\infty(P)$ and $S_{ij} \in \toeplitz^\infty(P)$ for all $i, j \in \Omega$.

\item\label{coords} Let $i, j \in \Omega$. 
There exists $m_0 \in \nn$ such that for all $m\geq m_0$, and
$B \in Arv(P)_m$ we have that
$$
S_{ij}(B) = E_{ij}B,  \qquad
U(B) = B
\qquad \text{and} \qquad
U^*(B) = B
$$
\item For all $i,j,t, k  \in \Omega$ we have
$S_{ij}S_{tk} - \delta_{jt} S_{ik} \in \mathcal{J}(\mathcal{T}^{\infty}(P))$
\item
$ U^*U - I, \ \ UU^* - I \in \mathcal{J}(\mathcal{T}^{\infty}(P))$
\item For all $i,j \in \Omega$ we have
$S_{ij} U - U S_{ij} \in \mathcal{J}(\mathcal{T}^{\infty}(P))$
\item The family $(\overline{U}, \{ \overline{S_{ij}}\}_{i,j \in \Omega})$ generates $\cuntz(P)$.
\end{enumerate}
Therefore, $\{\overline{S_{ij}} \}_{i,j\in \Omega}$ is a system of $d \times d$ matrix units in $\mathcal{O}(P)$ and $\overline{U}$ is a unitary in $\mathcal{O}(P)$ that commutes with them and together they generate $\mathcal{O}(P)$.

\end{lemma}

\begin{proof}\hfill

\begin{enumerate} 
\item 
Let $i,j\in \Omega$ be such that $i \leq j$ in the properly enumerated decomposition of $\Omega$, so that $\ell := \sigma(j) - \sigma(i) \geq 0$.
By item (2) of the cyclic decomposition Theorem \ref{theorem:cyclic-graph-decomposition}, there exists $n_0 \in \mathbb{N}$ such that $E_{ij} \in Arv(P)_{nr+\ell}$ and $I_d \in Arv(P)_{nr}$ for all $n \geq n_0$. Then for all $n\geq n_0$, $m\in \nn$ and $B\in Arv(P)_m$ we have
\begin{align*}
W_{I_d}^{(nr)*}W_{E_{ij}}^{(nr+\ell)}(B) & = Gr(P^{m+\ell})*(E_{ij}B) = S_{ij}(B) \\
W_{I_d}^{(nr)*}W_{I_d}^{(nr+r)}(B) & = Gr(P^{m+r})*B = U(B)
\end{align*}
so that
$$
S_{ij} = W_{I_d}^{(nr)*}W_{E_{ij}}^{(nr+\ell)} \ \text{ and } \ U = W_{I_d}^{(nr)*}W_{I_d}^{(nr +r)}
$$

\item Let $i \leq j \in \Omega$ be given. 
By the previous item and by Proposition~\ref{proposition:nograph} there exists $m_0 \in \nn$ such that for all $m \geq m_0$ and $B\in Arv(P)_m$ we have
\begin{align*}
S_{ij}(B) & = W_{I_d}^{(nr)*}W_{E_{ij}}^{(nr+\ell)} (B) = E_{ij}B  \\
U(B) & = W_{I_d}^{(nr)*}W_{I_d}^{(nr +r)}(B) = B
\end{align*}
Similarly, by taking adjoints, we have that
\begin{align*}
S_{ji}(B) & = W_{E_{ji}}^{(nr+\ell)*} W_{I_d}^{(nr)}(B) = E_{ji}B \\
U^*(B) & = W_{I_d}^{(nr +r)*}W_{I_d}^{(nr)}(B) = B
\end{align*}
proving the statement in all cases.

\item 
Let $i, j, t, k \in \Omega$ be given. By item~\eqref{coords}, there exists $m_0 \in \nn$ such that for all $m\geq m_0$ and $B\in Arv(P)_m$
we have that
$$
S_{ij} S_{tk} (B) = E_{ij} E_{tk} B = \delta_{jt} E_{ik} B = \delta_{jt} S_{ik}(B)
$$
Thus we have that $S_{ij}S_{tk} - \delta_{jt} S_{ik} \in \JJ(\toeplitz^\infty(P))$.

\item By item~\eqref{coords}, there exists $m_0 \in \nn$ such that for all $m\geq m_0$ and $B\in Arv(P)_m$ we have that
$$
U^*U(B) = U^*(B) = B \in Arv(P)_m \qquad \text{and} \qquad UU^*(B) = U^*(B) = B \in Arv(P)_m
$$
Thus we have that $U^*U - I, UU^*-I \in \JJ(\toeplitz^\infty(P))$

\item Let $i,j\in \Omega$ be given. By item~\eqref{coords}, there exists $m_0 \in \nn$ such that for all $m\geq m_0$ and $B\in Arv(P)_m$ we have the following element in $Arv(P)_{m+r+\ell}$ where $\ell = \sigma(j) - \sigma(i)$.
$$
S_{ij}U(B) = S_{ij}(B) = E_{ij}B = U(E_{ij}B) = US_{ij}(B)
$$
Thus we have that $S_{ij}U - US_{ij} \in \JJ(\toeplitz^\infty(P))$.

\item 
We first observe that since we are dealing with stochastic matrices over a \emph{finite} state space $\Omega$, it is in fact the case that $\mathcal{T}^{\infty}(P)$ is generated by $\ell^{\infty}(\Omega)$ and $\{W_{E_{ij}}^{(1)}\}_{(i,j)\in E(P)}$, where $E(P) = \{ \ (i,j) \ | \ P_{ij} >0 \ \}$. Indeed, since every $Arv(P)_n$ is finite dimensional, every $W_{A}^{(n)}$ can be written as a linear combination of elements of the form $W_{E_{ik}}^{(n)}$. Now, if $W_{E_{ik}}^{(n)}$ is non-zero, this means that $P^{(n)}_{ik} > 0$ and so there is a path of length $n$ from $i$ to $k$ given by $i=j_0 \rightarrow j_1 \rightarrow ... \rightarrow j_n = k$ and we would have that
$W_{E_{ik}}^{(n)} = W_{E_{j_0j_1}}^{(1)} \cdot ... \cdot W_{E_{j_{n-1}j_n}}^{(1)}$ so that every element $W^{(n)}_{E_{ik}}$ is in the algebra generated by $\ell^{\infty}(\Omega)$ and $\{W_{E_{ij}}^{(1)}\}_{(i,j) \in E(P)}$, and so 
$$
\mathcal{T}^{\infty}(P) = C^*\big(\ell^{\infty}(\Omega) \cup \{W_{E_{ij}}^{(1)}\}_{(i,j) \in E(P)}\big)
$$
Therefore $\mathcal{O}(P)$ is generated as a C*-algebra by the images of $\ell^{\infty}(\Omega)$ and $\{W_{E_{ij}}^{(1)} \}_{(i,j) \in E(P)}$ under $q:\mathcal{T}^{\infty}(P)\rightarrow \mathcal{O}(P)$.

Let us denote by $\cala$ the C*-subalgebra of $\cuntz(P) 
\cong \mathcal{T}^{\infty}(P) / \mathcal{J}(\mathcal{T}^{\infty}(P))$
generated by $\overline{U}$ and $\overline{S_{ij}}$ for $i,j \in \Omega$.
It follows from item~\eqref{coords} that $S_{ii} - p_i \in 
\JJ(\toeplitz^{\infty}(P))$, therefore, we have that $q(\ell^\infty(\Omega)) \subseteq \cala$. In order to complete
the proof that $\cala = \cuntz(P)$, it suffices to show that $\overline{W_{E_{ij}}^{(1)}} \in \cala$ for all $(i,j) \in E(P)$.

Let $(i,j) \in E(P)$, and suppose that $r>1$.
If $i \leq j$, then we must have by the cyclic decomposition theorem that $\sigma(j)- \sigma(i) = 1$ and 
$S_{ij}$ is an operator of degree one and by item~\eqref{coords}
we have that $S_{ij} - W_{E_{ij}}^{(1)} \in \mathcal{J}(\mathcal{T}^{\infty}(P))$. 
On the other hand, if $i > j$, then also by the cyclic decomposition theorem we must have
$\sigma(i)- \sigma(j) = r-1$ and  in that case $S_{ij}$ is an
operator of degree $-(r-1)$. Therefore $US_{ij}$ has degree $1$,
and by item~\eqref{coords} we have that $US_{ij} - W_{E_{ij}}^{(1)} \in \mathcal{J}(\mathcal{T}^{\infty}(P))$. Therefore, in both cases
we obtain that 
$\overline{W_{E_{ij}}^{(1)}} \in \cala$.

Finally, if $(i,j) \in E(P)$, and $r=1$, we have that $S_{ij}$ is
an operator of degree zero and by item~\eqref{coords} we have that $US_{ij} - W_{E_{ij}}^{(1)} \in \mathcal{J}(\mathcal{T}^{\infty}(P))$. Therefore, we also obtain that $\overline{W_{E_{ij}}^{(1)}} \in \cala$.
\end{enumerate}
\end{proof}

Recall from the discussion preceding \cite[Proposition 5.5]{dor-on-markiewicz} that there is a natural gauge group action $\alpha$ on $\mathcal{T}^{\infty}(P)$ uniquely determined by $\alpha_{\lambda}(W_A^{(n)}) = \lambda^n W_A^{(n)}$. Since $\mathcal{J}(\mathcal{T}^{\infty}(P))$ is gauge invariant by \cite[Theorem 5.6]{dor-on-markiewicz}, this gauge action passes to the quotient $\mathcal{O}(P)$, and we denote by $\mathcal{O}(P)_0$ the fixed point algebra.

\bigskip
\begin{proof}[\textbf{Proof of Theorem~\ref{theorem:main}}]
First note that $M_d(\mathbb{C}) \otimes C(\mathbb{T})$ is the universal C*-algebra generated by a system of $d \times d$ matrix units $e_{ij}$ and a unitary $u$ that commutes with them. Hence, by Lemma \ref{lemma:almost-universal} we obtain a surjective *-homomorphism $\psi : M_d(\mathbb{C}) \otimes C(\mathbb{T}) \rightarrow \mathcal{O}(P)$ that sends $e_{ij}$ to $\overline{S_{ij}}$ and $u$ to $\overline{U}$. It remains to show that $\psi$ is injective.

Let $\cala =\bigoplus_{\ell= 0}^{r-1}M_{|\Omega_{\ell}|}(\mathbb{C}) \otimes 1  \subseteq M_d(\mathbb{C}) \otimes C(\mathbb{T})$. 
First we note that $\psi$ restricted to $\cala$ is injective, since $\psi$ is already injective when restricted to the larger simple subalgebra $M_d(\cc) \otimes 1$. 

We now show that $\psi(\cala) = \cuntz(P)_0$. First note that 
$\cuntz(P)_0$ is generated by monomials of degree zero (according
to the gauge action) in the matrix units $(\overline{S_{ij}})$ and
the unitary $\overline{U}$, which commutes with the latter. 
Let $X \in \cuntz(P)_0$ be such a monomial. Products of matrix units are also matrix units, therefore there
exists $i,j\in \Omega$, $n\in\mathbb{Z}$ such that $X= \overline{S_{ij}}\, \overline{U}^n$. Hence,
the only way that $X$ has degree zero is if $n=0$ and $\sigma(i)=\sigma(j)$. Moreover,  $\cala$ is precisely generated
by all $e_{ij}$, $i,j\in \Omega$ such that $\sigma(i)=\sigma(j)$. Hence
$\psi(\cala)= \cuntz(P)_0$.

Next, we show $\psi$ is injective on the entire algebra $M_d(\cc)\otimes C(\torus)$. Given the identifications $M_d(\cc)\otimes C(\torus)
\cong C(\torus; M_d(\cc))$
and $M_d(\cc)\otimes 1 \cong M_d(\cc)$, let us consider the \emph{faithful} conditional expectation $\Gamma_0: M_d(\cc)\otimes C(\torus) \to M_d(\cc) \otimes 1$ given by
$$
\Gamma_0(T) = \int_{\torus} T(z)\, dz
$$
where $dz$ represents normalized Haar measure on the circle. Note that in
particular, for all $i,j \in \Omega$ and $n \in \mathbb{Z}$,
$$
\Gamma_0(e_{ij} u^n) = \delta_{0,n} \, e_{ij} 
$$
We now take $E_0$ to be the \emph{faithful} conditional expectation from $M_d(\mathbb{C}) \otimes 1$ to $\bigoplus_{\ell= 0}^{r-1}M_{|\Omega_{\ell}|}(\mathbb{C}) \otimes 1$, and let $\Phi_0: \cuntz(P) \to \cuntz(P)_0$ denote the canonical conditional expectation into the fixed point algebra associated
with the gauge action. We then have that $\Phi_0  \psi = \psi E_0 \Gamma_0$. Indeed, since for all $i,j \in \Omega, n\in\nn$,
\begin{align*}
\Phi_0 \psi(e_{ij} u^n) & = \Phi_0(\overline{S_{ij}}\,\overline{U}^n) 
= \delta_{0,n} \delta_{\sigma(i), \sigma(j)} \, \overline{S_{ij}} 
= \delta_{0,n} \delta_{\sigma(i), \sigma(j)} \, \psi(e_{ij} ) 
 = \delta_{0,n} \psi(E_0(e_{ij}) ) \\
 & = \psi(E_0(\Gamma_0(e_{ij} u^n))),
\end{align*}
and since monomials are total in the algebra, we have $\Phi_0 \circ \psi = \psi E_0 \Gamma_0$.

Finally, suppose towards a contradiction that $\psi$ is not injective. Then
there exists a positive non-zero $T \in M_d(\cc)\otimes C(\torus)$
such that $\psi(T) = 0$. In that case $\Phi_0(\psi(T))=0$. Hence 
$\psi(E_0(\Gamma_0(T))) = \Phi_0(\psi(T)) = 0$. By injectivity of
$\psi$ on the image of $E_0$, which is the algebra $\cala$, we 
obtain $E_0(\Gamma_0(T))=0$. We reach a contradiction since $E_0$
and $\Gamma_0$ are faithful conditional expectations.
\end{proof}

Now that we have filled the gap in the computation of the Cuntz-Pimsner algebra of a finite irreducible stochastic matrix, we compute the extension groups for it, which will be useful to us later in Section \ref{Sec:Classification}.

Based on the work of \cite{Paschke-Salinas}, one has a description of $\Ext_s(B \otimes M_d)$ for any unital C*-algebra $B$, for which $\Ext_s(B)$ contains no elements of order $d$, as follows. For any unital extension $\eta\in E(B \otimes M_d)$, we define a map $[\eta]_s \mapsto ([\iota_*\eta]_s,[j_*\eta]_s)$ into $\Ext_s(B) \otimes \mathbb{Z}_d$ by setting $\iota_*\eta = \eta |_{B \otimes I}$ and $j_*\eta = \eta|_{I \otimes M_d}$.
Then \cite[Proposition 2.2]{Paschke-Salinas} shows that this map induces an isomorphism of semigroups
$$
\Ext_s(B \otimes M_d) \cong \{ \ (d[\eta] + \epsilon(\ell)[\tau], \ell) \in \Ext_s(B) \otimes \mathbb{Z}_d \ | \ \eta \in E(B), \ \ell \in \mathbb{Z} \ \}
$$
where $\tau$ is a trivial strongly unital extension. By Example \ref{ex:torus-extensions} we have that $\epsilon(\ell)[\eta]_s = [\eta]_s$ for all $\eta \in \Ext_s(C(\mathbb{T}))$, so that
$$
\Ext_s(C(\mathbb{T}) \otimes M_d) \cong \{ (ds, \ell) \in \mathbb{Z} \times \mathbb{Z}_d \ | \ s \in \mathbb{Z}, \ \ell \in \mathbb{Z} \ \}
$$
so that $\Ext_s(C(\mathbb{T}) \otimes M_d) \cong d \mathbb{Z} \times \mathbb{Z}_d$ and $\Ext_w(C(\mathbb{T}) \otimes M_d) \cong \mathbb{Z}$ as the projection (and division by $d$) onto the first coordinate of $\Ext_s(C(\mathbb{T}) \otimes M_d)$. Since $\Ext_w(C(\mathbb{T}) \otimes M_d) \cong \mathbb{Z}$ is the quotient of $\Ext_s(C(\mathbb{T}) \otimes M_d) \cong d \mathbb{Z} \times \mathbb{Z}_d$ by the subgroup $\{ \ \epsilon(n)[\tau]_s \ | \ n\in \mathbb{Z} \ \} \cong \mathbb{Z}_d$, we can identify the subgroup $\{ \ \epsilon(n)[\tau]_s \ | \ n\in \mathbb{Z} \ \}$ of $\Ext_s(C(\mathbb{T}) \otimes M_d)$ with the image $\{ \ [j_*\eta]_s \ | \ [\eta]_s \in \Ext_s(C(\mathbb{T}) \otimes M_d) \ \} \cong \mathbb{Z}_d$.

Note that any automorphism $\beta$ of $C(\mathbb{T})\otimes M_d$ induces an automorphism $\beta_s$ of $\Ext_s(C(\mathbb{T}) \otimes M_d)$ by composition $[\eta]_s \mapsto [\eta \circ \beta]_s$. Furthermore, every unitary element $u\in \mathcal{U}(C(\mathbb{T}) \otimes M_d)$ defines an automorphism $Ad_u$ of $C(\mathbb{T})\otimes M_d$ by way of $Ad_u(f)(z) = u^*(z)f(z)u(z)$ for $f\in C(\mathbb{T}; M_d)$ and $z\in \mathbb{T}$. Denote by $Aut_{C(\mathbb{T})}(C(\mathbb{T})\otimes M_d)$ the collection of $C(\mathbb{T})$-bimodule *-automorphisms of $C(\mathbb{T})\otimes M_d$.

\begin{prop} \label{prop:automorphism-invert}
Let $\eta$ be a unital extension and let $\beta \in Aut(C(\mathbb{T})\otimes M_d)$ be an automorphism. Up to the identification $\Ext_s(C(\mathbb{T})\otimes M_d) \cong d\mathbb{Z} \otimes \mathbb{Z}_d$ given above, we have that either $\beta_s[\eta] = [\eta] = ([\iota_*\eta],[j_*\eta])$ or $\beta_s[\eta] = (-[\iota_*\eta],[j_*\eta])$.
\end{prop}

\begin{proof}
Let $\beta \in Aut(C(\mathbb{T}) \otimes M_d)$ be some *-automorphism. Then $\beta$ induces an automorphism $\beta_*$ on the primitive ideal spectrum $\mathbb{T}$, which then induces an automorphism $(\beta_*)^*$ back on $C(\mathbb{T}) \otimes M_d$ given by $(\beta_*)^*(f)(z) = f(\beta_*^{-1}(z))$. It is easy to see that $[j_*\eta] = [j_* \eta \circ (\beta_*)^*]$ since $(\beta_*)^*(I\otimes M_d) = I \otimes M_d$. Since the induced map $((\beta_*)^*)_s$ on $\Ext_s(C(\mathbb{T})\otimes M_d) \cong d\mathbb{Z} \times \mathbb{Z}_d$ is the identity on the second coordinate $\mathbb{Z}_d$, we must have that $[\iota_* (\eta \circ (\beta_*)^*)]$ is either $[\iota_*\eta]$ or $-[\iota_*\eta]$. Hence, by composing with the inverse of $(\beta_*)^*$ if necessary, we may assume that ${\beta}_* = Id_{\mathbb{T}}$. 

By \cite[Corollary 5.46]{Morita-continuous-trace} we have that $\beta \in Aut_{C(\mathbb{T})}(C(\mathbb{T}) \otimes M_d)$, so that by \cite[Lemma 4.28]{Morita-continuous-trace}, there is a point-norm continuous map $\sigma : \mathbb{T} \rightarrow Aut(M_d)$ such that $\beta(f)(z) = \sigma_z(f(z))$. Since the second cohomology group of the torus $H^2(\mathbb{T} ; \mathbb{Z})$ vanishes, by \cite[Theorem 5.42]{Morita-continuous-trace}, there is a unitary element $u \in \mathcal{U}(C(\mathbb{T}) \otimes M_d)$ such that $\beta = Ad_u$.
Then $Ad_u$ induces a map on $\Ext_s(C(\mathbb{T})\otimes M_d)$, so that by the homomorphism property of the Fredholm index, we get that,
$$
[\iota_*\eta \circ Ad_u] = \ind(\eta(Ad_u(z\otimes I))) = \ind(\eta(z \otimes I)) = [\iota_* \eta]
$$

Next, since the image $\{ \ [j_*\eta]_s \ | \ [\eta]_s \in \Ext_s(C(\mathbb{T}) \otimes M_d) \ \} \cong \mathbb{Z}_d$ can be identified with the subgroup $\{ \ \epsilon(n)[\tau]_s \ | \ n\in \mathbb{Z} \ \}$, in order to show that $[j_*\eta \circ \beta] = [j_*\eta]$, it will suffice to show that $\beta_s(\epsilon(n)[\tau]_s) = \epsilon(n)[\tau]_s$. However, since $\beta_s$ commutes with $\epsilon(n)$, it will suffice to show that $\beta_s([\tau]_s) = [\tau]_s$. But $\beta_s$ is a group homomorphism, so it must send $[\tau]_s$ to itself. Hence, we obtain that $[j_*\eta \circ \beta] = [j_*\eta]$.
\end{proof}

\section{Non-commutative Choquet boundary} \label{Sec:Choquet}

In this section we first find all the irreducible representations of $\toeplitz(P)$ for a finite irreducible stochastic matrix $P$. We then determine the boundary representations with respect to $\tensor(P)$ among them. We show that any representation annihilating $\JJ(P):=\JJ(\toeplitz(P))$ has the unique extension property when restricted to $\tensor(P)$, and find conditions that guarantee that an irreducible representation supported on $\JJ(P)$ is boundary or not.

As given in \cite[Theorem 5.6]{dor-on-markiewicz}, the C*-algebra $\toeplitz^{c}(P)$ is the one generated by both $\toeplitz^{\infty}(P)$ and $\toeplitz(P)$, and it too has a gauge action which is the restriction of the gauge action of $\LL(\FF_{Arv(P)})$, which satisfies $\alpha_{\lambda}(S_A^{(n)}) = \lambda^nS_A^{(n)}$ and $\alpha_{\lambda}(W_A^{(n)}) = \lambda^nW_A^{(n)}$, so that $\toeplitz^c(P)$ is gauge invariant, and $\JJ(\toeplitz^c(P))$ is a closed gauge invariant two-sided ideal by \cite[Theorem 5.6]{dor-on-markiewicz}. 

As discussed in \cite[Section 4]{dor-on-markiewicz} for general subproduct systems, Fourier coeficients $\Phi_k$ on $\toeplitz^c(P)$ may be defined in such a way that every $T\in \toeplitz^c(P)$ can be written as $\sum_{k= -\infty}^{\infty}\Phi_k(T)$, where this sum convergens Cesaro. That is, where $\sum_{k=-n}^n\big(1 - \frac{|k|}{n+1} \big)\Phi_k(T)$ converges in norm to $T$. From now on, we will denote $W_{ij}:= W_{E_{ij}}$ for $i,j\in \Omega$.

\begin{prop} \label{prop:alt-desc-for-cuntz-ideal}
Let $P$ be an irreducible stochastic matrix on $\Omega$ of size $d$. Then 
$\JJ(\toeplitz^c(P))$ is the two sided ideal generated by $\{ Q_n \}_{n \in \nn}$ inside $\toeplitz^c(P)$.
\end{prop}

\begin{proof}
By \cite[Proposition 5.2]{dor-on-markiewicz} we see that $Q_n \in \toeplitz(P) \subset \toeplitz^c(P)$, and since $\|Q_n Q_m\| \rightarrow 0$ as $m$ goes to infinity, we see that $Q_n \in \JJ(\toeplitz^c(P))$. 

For the reverse inclusion, let $T\in \JJ(\toeplitz^c(P))$, and write $T = \sum_{k = -\infty}^{\infty} \Phi_k(T)$ as a Cesaro convergent sum where $\Phi_k(T)$ maps $Arv(P)_n$ to $Arv(P)_{n+k}$ if $n+k \geq 0$ and $\{0\}$ otherwise. Further notice that $\Phi_k(T) \in \JJ(\toeplitz^c(P))$ for all $k \in \mathbb{Z}$, since by \cite[Theorem 5.6]{dor-on-markiewicz} we have that $\JJ(\toeplitz^c(P))$ is gauge invariant. In this case, we have that $\| \Phi_k(T) Q_{[n+1 , \infty)} \| = \sup_{m\geq n+1}\| \Phi_k(T)Q_m\| \rightarrow 0$. Hence, since $\Phi_k(T) Q_{[0 , n]}$ is in the ideal generated by $\{ Q_n \}_{n\in \nn}$, we see that $\Phi_k(T)$ is in the closed ideal generated by $\{ Q_n \}_{n \in \nn}$ and so must be $T$ by Cesaro approximation.
\end{proof}

For a finite irreducible stochastic matrix $P$ with state set $\Omega$ of size $d$, we have that $\ell^{\infty}(\Omega)$ is faithfully represented in $B(\ell^2(\Omega))$ by diagonal matrix multiplication on columns. Hence by \cite[Corollary 2.74]{Morita-continuous-trace}, this faithful *-representation promotes to a faithful *-representation $\pi : \LL(\FF_{Arv(P)}) \rightarrow B(\FF_{Arv(P)} \otimes_{id}\ell^2(\Omega))$ given by $\pi(T) (\xi \otimes h) = T\xi \otimes h$. Note that $\FF_{Arv(P)} \otimes_{id} \mathbb{C} e_k$ is a reducing subspace for $\pi(\toeplitz^c(P))$ for each $k \in \Omega$.

\begin{notation} \label{notation:state-invariant-fock}
For a state $k\in \Omega$ we will find it useful to denote $Arv(P)_{n,k}:= Arv(P)_n \otimes \mathbb{C} e_k$, and $\FF_{P,k} : = \oplus_{n=0}^{\infty} Arv(P)_{n,k} = \FF_{Arv(P)} \otimes_{id} \mathbb{C} e_k$, the reducing Hilbert space for $\pi(\toeplitz^c(P))$ mentioned above, so that $\FF_{Arv(P)} \otimes \ell^2(\Omega) = \oplus_{k\in \Omega} \FF_{P,k}$. For fixed $n$ we also denote for $i\in \Omega$ with $(i,k)\in Gr(P^n)$ the elements $e_{ik}^{(n)}:= E_{ik} \otimes e_k \in Arv(P)_{n,k}$ which comprise a finite orthonormal basis for each $Arv(P)_{n,k}$, so that for varying $n\in \mathbb{N}$ and $i\in \Omega$ with $(i,k)\in E(P^n)$ the collection $\{e_{ik}^{(n)} \}$ is an orthonormal basis for $\FF_{P,k}$.
\end{notation}

\begin{prop} \label{prop:irreducible-cuntz}
Let $P$ be an irreducible stochastic matrix over $\Omega$ of size $d$. Then for each $\pi_k : \toeplitz^c(P) \rightarrow B(\FF_{P,k})$ given by $\pi_k(T) = \pi(T) |_{\FF_{P,k}}$ we have that $\pi_k(\toeplitz(P))$ is an irreducible subalgebra of $B(\FF_{P,k})$.
\end{prop}

\begin{proof}
By \cite[Proposition 5.2]{dor-on-markiewicz} we see that $Q_n\in \toeplitz(P)$ for every $n\in \nn$. Let $0 \neq H' \subseteq \FF_{P,k}$ be some non-zero invariant subspace. Since $\{\pi_k(Q_{[0,n]})\}$ converges SOT to the identity on $\FF_{P,k}$, there is some minimal $n_0 \in \nn$ such that $\pi_k(Q_{n_0}) \xi \neq 0$ for some $\xi \in H'$. In this case, $0 \neq \pi_k(Q_{n_0}) \xi = A \otimes e_k \in H' \cap Arv(P)_{n_0,k}$ for some $A \in Arv(P)_{n_0}$, so that there exists $j\in \Omega$ and some non-zero scalar $c\in \mathbb{C}$ with $0 \neq e^{(n_0)}_{jk} = c \cdot \pi_k(p_jQ_{n_0})\xi \in H'$ where $(j,k) \in E(P^{n_0})$. This means that $e^{(0)}_{kk}= c_1 \pi_k(S^{({n_0})*}_{E_{jk}})(e^{(n_0)}_{jk}) \in H'$, where $c_1 > 0 $ is some scalar. 

Thus, for $m \geq 0$ if $e^{(m)}_{ik}$ is some vector in $Arv(P)_{m,k}$, we see that $e^{(m)}_{ik} = c_2 \pi_k(S^{(m)}_{E_{ik}})(e^{(0)}_{kk}) \in H'$ where $c_2 >0$ is some scalar.
This shows that the set of elements $e^{(m)}_{ik}$ for all $m\geq 0$ and $(i,k) \in E(P^m)$ is in $H'$, and this set of elements is an orthonormal basis for $\FF_{P,k}$, and so $H' = \FF_{P,k}$.
\end{proof}

Hence, we see that $\pi$ decomposes into $d = |\Omega|$ irreducible representations $\pi_k$ as above, so that $\pi = \oplus_{k\in \Omega}\pi_k : \toeplitz^c(P) \rightarrow \oplus_{k\in \Omega} B(\FF_{P,k})$. We next show that each $\pi_k |_{\toeplitz(P)}$ is in a distinct unitary equivalence class of irreducible representations for $\toeplitz(P)$.

\begin{prop} \label{prop:equivalent-representations}
Let $P$ be a finite irreducible stochastic matrix on $\Omega$ and $k, k' \in \Omega$ be \emph{distinct} indices. Then $\pi_k |_{\toeplitz(P)}$ and $\pi_{k'}|_{\toeplitz(P)}$ are not unitarily equivalent.
\end{prop}

\begin{proof}
Suppose that $k,k' \in \Omega$ are such that $\pi_k |_{\toeplitz(P)}$ and $\pi_{k'}|_{\toeplitz(P)}$ are unitarily equivalent. Then there is a unitary $U: \FF_{P,k} \rightarrow \FF_{P,k'}$ such that $U \pi_k(T) = \pi_{k'}(T)U$ for all $T\in \toeplitz(P)$. For $j\in \Omega$, we have that $p_j Q_0 \in \toeplitz(P)$, so that
$U \pi_k(p_j Q_0) = \pi_{k'}(p_j Q_0)U$. Apply this operator to $e^{(0)}_{kk} \in Arv(P)_{0,k} \subset \FF_{P,k}$ and get
$$
\pi_{k'}(p_jQ_0)U(e^{(0)}_{kk}) = U \pi_{k}(p_j Q_0)(e^{(0)}_{kk}) = U(\delta_{jk} e^{(0)}_{kk})
$$
On the other hand $\pi_{k'}(Q_0)U(e^{(0)}_{kk})$ must have image in $Arv(P)_{0,k'}$ so that $\pi_{k'}(Q_0)U(e^{(0)}_{kk}) = c \cdot e^{(0)}_{k'k'}$ for some non-zero $c \in \cc$. But after applying $\pi_{k'}(p_j)$ we would obtain that 
$$
\pi_{k'}(p_jQ_0)U(e^{(0)}_{kk}) = c \cdot \pi_{k'}(p_j)(e^{(0)}_{k'k'}) = c \cdot \delta_{jk'} e^{(0)}_{k'k'}
$$
Thus, we see that if $k \neq k'$ then by taking $j=k$ we would obtain that $0 = c \cdot \delta_{jk'} e^{(0)}_{k'k'} = U(\delta_{jk} e^{(0)}_{kk}) \neq 0$ in contradiction. Hence, $\pi_k |_{\toeplitz(P)}$ and $\pi_{k'}|_{\toeplitz(P)}$ are not unitarily equivalent. 
\end{proof}

\begin{prop} \label{prop:toeplitz-char}
Let $P$ be a finite irreducible stochastic matrix on $\Omega$. Then $\JJ(\toeplitz(P)) = \JJ(\toeplitz^{c}(P))$ and is *-isomorphic to $\oplus_{k\in \Omega}\KK(\FF_{P,k})$. Thus, we have that $\toeplitz^{\infty}(P) \subseteq \toeplitz^c(P) = \toeplitz(P)$.
\end{prop}

\begin{proof}
By Proposition \ref{prop:alt-desc-for-cuntz-ideal}, we have that $\JJ(\toeplitz^c(P))$ is the ideal generated by $\{Q_n \}_{n\in \nn}$ inside $\toeplitz^c(P)$, and since $\pi(Q_n)$ is a finite rank operator, we see by Proposition \ref{prop:irreducible-cuntz} that $\pi_k(\JJ(\toeplitz^c(P)))$ and $\pi_k(\JJ(\toeplitz(P)))$ are irreducible compact operator subalgebras of $B(\FF_{P,k})$ and hence by \cite[Theorem 1.3.4]{arv-book} they must both be equal to $\KK(\FF_{P,k})$. Write the identity representation $Id : \pi(\JJ(\toeplitz^c(P))) \rightarrow \oplus_{k\in \Omega} B(\FF_{P,k})$ as a direct sum of irreducible representations with multiplicity $Id = \bigoplus \ n(\zeta) \cdot \zeta$, where each $\zeta$ is a representative in the equivalence class of irreducible representation given by restriction to some $\FF_{P,k}$ for some $k$. Then by Proposition \ref{prop:equivalent-representations} we have that $n(\zeta) = 1$ for all $\zeta$ and that $Id |_{\pi(\JJ(\toeplitz(P)))}$ has the same decomposition into irreducible representations as the one above. Since $\pi = \bigoplus \pi_k$ is injective on $\JJ(\toeplitz^c(P))$, we have that $\pi(\JJ(\toeplitz^c(P))) = \oplus_{k \in \Omega} \KK(\FF_{P,k}) = \pi(\JJ(\toeplitz(P)))$, and by taking the inverse of the faithful *-representation $\pi$, we obtain $\JJ(\toeplitz(P)) = \JJ(\toeplitz^c(P))$.

Finally, by \cite[Proposition 5.5]{dor-on-markiewicz} we have that $\toeplitz(P) = \toeplitz(P) + \JJ(\toeplitz(P)) = \toeplitz(P) + \JJ(\toeplitz^c(P)) = \toeplitz^c(P)$ so that $\toeplitz^{\infty}(P) \subseteq \toeplitz^c(P) = \toeplitz(P)$.
\end{proof}

We next wish to parametrize all irreducible representations of $\toeplitz(P)$. Under the identification $\cuntz(P) \cong C(\mathbb{T}, M_d)$ and $\JJ(\toeplitz(P)) \cong \oplus_{k\in \Omega}\KK(\FF_{P,k})$ we have the following exact sequence
$$
0 \rightarrow \oplus_{k\in \Omega}\KK(\FF_{P,k}) \rightarrow \toeplitz(P) \rightarrow C(\mathbb{T}, M_d) \rightarrow 0
$$
If $\rho : \toeplitz(P) \rightarrow B(H)$ is a unital representation, by the discussion preceding \cite[Theorem 1.3.4]{arv-book} it decomposes uniquely into a central direct sum of representations $\rho = \rho_{\JJ} \oplus \rho_{\mathcal{O}}$, where $\rho_{\JJ}$ is the unique extension to $\toeplitz(P)$ of the restriction of $\rho$ to $\JJ(\toeplitz(P))$, and $\rho_{\mathcal{O}}$ annihilates $\JJ(\toeplitz(P))$. Hence, the spectrum of $\toeplitz(P)$ decomposes into a disjoint union of the spectrum of $\JJ(\toeplitz(P)) \cong \oplus_{k\in \Omega}\KK(\FF_{P,k})$ and the spectrum of $\cuntz(P) \cong C(\mathbb{T},M_d)$.

For $\lambda \in \mathbb{T}$, we define $ev_{\lambda} : C(\mathbb{T}, M_d) \rightarrow M_d$ given by $ev_{\lambda}([f_{ij}]) = [f_{ij}(\lambda)]$. Since $ev_{\lambda}$ has range $M_d$, we obtain that $ev_{\lambda} \circ q$ is an irreducible representation of $\toeplitz(P)$ where $q: \toeplitz(P) \rightarrow \cuntz(P)$ is the quotient map. Note that every $ev_{\lambda} \circ q$ is a $d$ dimensional representation.

\begin{cor} \label{cor:irreps-toeplitz}
Let $P$ be an irreducible stochastic matrix over $\Omega$ of size $d$. Then the spectrum of $\toeplitz(P)$ is parameterized by $d$ irreducible representations of infinite dimension, each unitarily equivalent to some $\pi_k$, and a torus $\torus$ of irreducible representations of dimension $d$ that annihilate $\JJ(\toeplitz(P))$, each unitarily equivalent to $ev_{\lambda} \circ q$ for some $\lambda \in \mathbb{T}$.
\end{cor}

\begin{proof}
If $\rho$ is an irreducible representation of $\toeplitz(P)$ that does not annihilate $\JJ(\toeplitz(P))$, we have  by \cite[Theorem 1.3.4]{arv-book} that $\rho |_{\JJ(\toeplitz(P))}$ is also irreducible. We use $\pi^{-1}$ to obtain an irreducible representation $\rho \circ \pi^{-1}$ of $\pi(\JJ(\toeplitz(P)))$. Since $\pi(\JJ(\toeplitz(P)))$ is a C*-algebra of compact operators, by \cite[Theorem 1.4.4]{arv-book} every irreducible representation of it is unitarily equivalent to some restriction to some $\FF_{P,k}$. Pushing this back via $\pi$ we obtain that $\rho$ is unitarily equivalent to some $\pi_k$.

For the other part, if $\rho$ does annihilate $\JJ(\toeplitz(P))$, it induces an irreducible representation of $\cuntz(P) \cong C(\mathbb{T}, M_d)$ by taking the quotient by $\JJ(\toeplitz(P))$. Since the irreducible representations of $C(\mathbb{T})$ are just point evaluations, and since $C(\mathbb{T})$ is strongly Morita equivalent to $C(\mathbb{T}, M_d)$,  we see that $\rho$ must be unitarily equivalent to the composition $ev_{\lambda} \circ q$ of an evaluation $ev_{\lambda} : C(\mathbb{T}, M_d) \rightarrow M_d$ given by $ev_{\lambda}([f_{ij}]) = [f_{ij}(\lambda)]$ and the natural quotient map $q : \toeplitz(P) \rightarrow \cuntz(P)$.

Thus, the spectrum of $\toeplitz(P)$ is parametrized by $d$ irreducible representations of infinite dimension, and a torus $\mathbb{T}$ of irreducible representations of dimension $d$.
\end{proof}

\begin{lemma} \label{lemma:eventually-projection}
Let $P$ be an irreducible stochastic matrix over a finite set $\Omega$, and let $\epsilon > 0$. There exists $m \geq 1$ and $M > 0$ such that for every $(i,j)\in E(P)$ we have
$$
(1+\epsilon)p_j  \geq T_{E_{ij}}^{(1)*}T_{E_{ij}}^{(1)} - M\cdot Q_{[0,m]}
$$
\end{lemma}

\begin{proof}
For $E_{k\ell} \in Arv(P)_m$ and $m\geq 1$, by definition of $T_{E_{ij}}^{(1)}$, we see that 
$$
T_{E_{ij}}^{(1)}(E_{k\ell}) = \delta_{j,k} \sqrt{\frac{P^{(m)}_{k\ell}}{P^{(m+1)}_{i\ell}}} E_{i\ell} \ \ \text{and} \ \ 
T_{E_{ij}}^{(1)*}(E_{k\ell}) = \delta_{i,k} \sqrt{\frac{P^{(m)}_{j \ell}}{P^{(m+1)}_{k\ell}}}E_{j\ell}
$$
So that
$$
T_{E_{ij}}^{(1)*}T_{E_{ij}}^{(1)}(E_{k\ell}) = \delta_{j,k} \sqrt{\frac{P^{(m)}_{k\ell}}{P^{(m+1)}_{i\ell}}} T_{E_{ij}}^{(1)*}(E_{i\ell}) = \frac{P^{(m)}_{j\ell}}{P^{(m+1)}_{i\ell}} p_j(E_{k\ell})
$$
By Theorem \ref{theorem:convergence-theorem-finite}, there exists $m$ such that $\frac{P^{(m)}_{j\ell}}{P^{(m+1)}_{i\ell}} \leq 1+ \epsilon$ for all $(i,j) \in E(P)$ and $\ell \in \Omega$ such that $(j,\ell) \in E(P^m)$. Hence, if we take $M = \|Q_{[0,m]} T_{E_{ij}}^{(1)*}T_{E_{ij}}^{(1)} \|$ it follows that
$(1+\epsilon)p_j  \geq T_{E_{ij}}^{(1)*}T_{E_{ij}}^{(1)} - M\cdot Q_{[0,m]}$ as required.
\end{proof}

We next show that representations annihilating $\JJ(P)$ have u.e.p. when restricted to $\tensor(P)$.

\begin{prop} \label{prop:annihi-auto-uep}
Let $P$ be a finite irreducible stochastic matrix over $\Omega$, and let $\rho : \toeplitz(P) \rightarrow B(H)$ be a *-representation such that $\rho(\JJ(P)) = \{0\}$. Then $\rho|_{\tensor(P)}$ has u.e.p.
\end{prop}

\begin{proof}
Let $\widetilde{\rho} : \tensor(P) \rightarrow B(K)$ be a maximal dilation of $\rho |_{\tensor(P)}$ such that $H$ is a subspace of $K$, and let $\psi : \toeplitz(P) \rightarrow B(K)$ be its (unique) extension to a *-representation. Denote
$$
\psi(p_i) = \begin{bmatrix}
\rho(p_i) & X_i \\
Y_i & Z_i
\end{bmatrix} \ \ \text{and} \ \ 
\psi(T_{E_{ij}}^{(1)}) = \begin{bmatrix}
\rho(T_{E_{ij}}^{(1)}) & X_{ij} \\
Y_{ij} & Z_{ij}
\end{bmatrix}
$$
First note that since $p_i$ is a self-adjoint projection, we get that
$$
\begin{bmatrix}
\rho(p_i) & X_i \\
Y_i & Z_i
\end{bmatrix} = \psi(p_i) = \psi(p_i)\psi(p_i)^* = \begin{bmatrix}
\rho(p_i) \rho(p_i)^* + X_iX_i^* & * \\
* & *
\end{bmatrix}
$$
So that by taking the $(1,1)$ compression, we obtain that $X_iX_i^* = 0$, so that $X_i = 0$. Now, since $\psi(p_i)$ is self-adjoint, we see that we must also have that $Y_i = 0$.

Next, for $(i,j)\in E(P^m)$, Suppose 
$$
\psi(S^{(m)}_{E_{ij}}) =\begin{bmatrix}
\rho(S^{(m)}_{E_{ij}}) & X(m)_{ij} \\
Y(m)_{ij} & Z(m)_{ij}
\end{bmatrix}
$$
Observe that for all $m\geq 1$, by the proof of \cite[Proposition 5.5]{dor-on-markiewicz} we have that
$$
0 \leq Q_{[0,m-1]} = Id - \sum_{(i,j)\in E(P^{m})}S^{(m)}_{E_{ij}}S^{(m)*}_{E_{ij}}
$$
Hence, by applying $\psi$ to this equation, we obtain that
$$
0 \leq \psi(Q_{[0,m-1]}) = Id - \sum_{(i,j)\in E(P^{m})}\psi(S^{(m)}_{E_{ij}})\psi(S^{(m)}_{E_{ij}})^*
$$
Then by compressing to the $(1,1)$ corner we get
$$
0 \leq Id - \sum_{(i,j)\in E(P^{m})} \big[\rho(S^{(m)}_{E_{ij}}) \rho(S^{(m)}_{E_{ij}})^* + X(m)_{ij}X(m)_{ij}^* \big] = - \sum_{(i,j)\in E(P^{m})} X(m)_{ij}X(m)_{ij}^*
$$
where the last equality follows due to the fact that $\rho$ annihilates $\JJ$.
Hence we must have that $X(m)_{ij} = 0$ for all $(i,j)\in E(P^m)$, so that the $(1,1)$ compression of $\psi(Q_{[0,m]})$ is $0$, and if we specify $m=1$, and note that $S^{(1)}_{E_{ij}} = \sqrt{P_{ij}}\cdot T^{(1)}_{E_{ij}}$, the above also yields that $X_{ij} = 0$ for all $(i,j)\in E(P)$.

Next, let $\epsilon > 0$. By Lemma \ref{lemma:eventually-projection} there exists $m\geq 1$ and $M > 0 $ such that for all $(i,j)\in E(P)$ we have that 
$$
(1+\epsilon)p_j  \geq T_{E_{ij}}^{(1)*}T_{E_{ij}}^{(1)} - M\cdot Q_{[0,m]}
$$
Hence, 
$$
(1+\epsilon)\psi(p_j) \geq \psi(T_{E_{ij}}^{(1)})^*\psi(T_{E_{ij}}^{(1)}) - M\cdot \psi(Q_{[0,m]})
$$
By compressing to the $(1,1)$ corner, we obtain that
$$
(1+\epsilon)\rho(p_j) \geq \rho(T_{E_{ij}}^{(1)})^*\rho(T_{E_{ij}}^{(1)}) + Y_{ij}^*Y_{ij}
$$
but $\rho(T_{E_{ij}}^{(1)})^*\rho(T_{E_{ij}}^{(1)}) = \rho(W^{(1)*}_{E_{ij}}W^{(1)}_{E_{ij}}) = \rho(p_j)$, so for every $\epsilon > 0$ we have that $\epsilon \cdot \rho(p_j) \geq Y_{ij}^*Y_{ij}$. Hence we have that $Y_{ij} = 0$ for all $(i,j)\in E(P)$.

Since $\toeplitz(P)$ is generated by $\{p_i\}_{i\in \Omega}$ and $\{T_{E_{ij}}^{(1)}\}_{(i,j)\in E(P)}$, we must have that $\psi$ has $\rho$ as a direct summand, so that $\widetilde{\rho}$ is a trivial dilation of $\rho |_{\tensor(P)}$, and hence $\rho |_{\tensor(P)}$ is maximal, and must then have the unique extension property.
\end{proof}

We next define a notion that will help us detect when an irreducible $\pi_k$ is not a boundary representation for $\tensor(P)$.

\begin{defi}
Let $P$ be a finite irreducible stochastic matrix over $\Omega$. A state $k\in \Omega$ is called \emph{exclusive} if whenever for $i\in \Omega$ and $n \in \nn$ we have $P^{(n)}_{ik} > 0$, then $P^{(n)}_{ik} = 1$. We denote by $\Omega_e$ the set of all exclusive states in $\Omega$.
\end{defi}

One should think of exclusive states as those states $k$ such that for any $n$ for which $i$ leads to $k$ in $n$ steps, it cannot lead anywhere else in $n$ steps.

\begin{lemma} \label{lemma:non-exclusive-char}
Let $P$ be a finite irreducible $r$-periodic stochastic matrix over $\Omega$, and $\Omega_0,..., \Omega_{r-1}$ be a cyclic decomposition for $P$. Suppose that $k \in \Omega_0$ is a state.

\begin{enumerate}
\item
$|\Omega_0| > 1$ if and only if $k$ is non-exclusive. In this case, any state in $\Omega_0$ is non-exclusive and there is an $n_0$ such that for any $n \geq n_0$ and $i,j\in \Omega_0$ we have $0 < P^{(rn)}_{ij} < 1$. 

\item
Assume $k$ is non-exclusive and $k\neq s \in \Omega$ is some different state. If there is $k \neq k' \in \Omega_0$ such that $P^{(m)}_{k's}> 0$ whenever $P^{(m)}_{ks} > 0$ for all $m\in \nn$, then there exists $n\in \nn$ such that $0 < P^{(rn)}_{kk} < 1$ and for all $m\in \nn$ with $(k,s) \in E(P^m)$ we have $P^{(rn)}_{kk}P^{(m)}_{ks} < P^{(rn+m)}_{ks}$.

\end{enumerate}
\end{lemma}

\begin{proof}
(1): Suppose $| \Omega_0| > 1$ and let $k \in \Omega_0$. By item (2) of Theorem \ref{theorem:cyclic-graph-decomposition} there is $n_0$ such that for all $n\geq n_0$ we would have that $P^{(rn)}_{ij} > 0$ for all $i,j\in \Omega_0$ and $n \geq n_0$. Thus, for some $j\in \Omega_0$ we have that $P^{(nr)}_{jj}, P^{(nr)}_{jk} > 0$, and since the $j$-th row sums up to $1$ we get that $0 < P^{(nr)}_{jk} <1$, and we conclude that $k$ is non-exclusive.

For the converse, suppose $k$ is non-exclusive. We show that $|\Omega_0| > 1$. Let $k' \in \Omega$ and $n_0$ be so that $0 < P^{(n_0)}_{k'k} < 1$, and let $m_0$ be large enough so that $P^{(m_0)}_{kk'} > 0$. Then
$$
P^{(m_0 + n_0)}_{kk} = \sum_{j\in \Omega} P^{(m_0)}_{kj}P^{(n_0)}_{jk} < \sum_{j\in \Omega} P^{(m_0)}_{kj} = 1
$$
On the other hand,
$$
P^{(m_0 + n_0)}_{kk} = \sum_{j\in \Omega} P^{(m_0)}_{kj}P^{(n_0)}_{jk} \geq P^{(m_0)}_{kk'}P^{(n_0)}_{k'k} > 0
$$
So we see that $0 < P^{(m_0 + n_0)}_{kk} < 1$. Since the $k$-th row sums up to $1$, there must be an $i\in \Omega$ differnt from $k$ such that $P^{(m_0 + n_0)}_{ki} > 0$ and by definition of the cyclic decomposition we have that $i\in \Omega_0$. This shows that $|\Omega_0| > 1$. 

Next, by item (2) of Theorem \ref{theorem:cyclic-graph-decomposition} we may find $n_0$ large enough so that for any $n \geq n_0$ we would have $P^{(rn)}_{ij} > 0$ for all $i,j \in \Omega_0$. As $| \Omega_0 | > 1$, and all rows sum up to $1$, we must also have that $ P^{(rn)}_{ij} < 1$ for all $i,j \in \Omega_0$. Hence, we see that all states in $\Omega_0$ are non-exclusive.

(2): By item (1) we can find $n_0$ so that $0 < P^{(rn)}_{ij} < 1$ for all $i,j \in \Omega_0$ and $n \geq n_0$. Now fix $m\in \nn$ with $P^{(m)}_{ks} > 0$, so that by assumption $P^{(m)}_{k's} > 0$. Then
$$
P^{(rn+m)}_{ks} = \sum_{i \in \Omega}P^{(rn)}_{ki}P^{(m)}_{is} \geq P^{(rn)}_{kk'}P^{(m)}_{k's} + P^{(rn)}_{kk}P^{(m)}_{ks} > P^{(rn)}_{kk}P^{(m)}_{ks}
$$

\end{proof}

\begin{prop}\label{prop:boundary-rep-exculsive}
Let $P$ be a finite $r$-periodic irreducible matrix over $\Omega$ and $\Omega_0,...,\Omega_{r-1}$ a cyclic decomposition for $P$. Let $k\in \Omega$. 
\begin{enumerate}
\item
If $k \in \Omega_0$ is non-exclusive and for any other non-exclusive $s \neq k$ there is some $k \neq k' \in \Omega_0$ such that $P^{(m)}_{k's} > 0$ whenever $P^{(m)}_{ks} > 0$, then $\pi_k$ is a boundary representation.
\item
If $k$ is exclusive then $\pi_k$ is not a boundary representation.

\end{enumerate}
\end{prop}

\begin{proof}
(1): Assume $k$ is non-exclusive. We use \cite[Theorem 7.2]{arveson-non-commutative-choquet-II} to show that $\pi_k$ is a strongly peaking representation according to \cite[Definition 7.1]{arveson-non-commutative-choquet-II}. Since the irreducible representations of $\toeplitz(P)$ are given by Corollary \ref{cor:irreps-toeplitz}, it suffices to find an element $T \in \tensor(P)$ such that $\| \pi_k(T) \| > \| (ev_{\lambda} \circ q)(T) \|$ for any $\lambda \in \mathbb{T}$ and such that $\| \pi_k(T) \| > \| \pi_{s}(T) \|$ for any $k\neq s$.

Choose $T = T^{(n)}_{E_{kk}}$, and wait until prescribing $n$ is necessary. Recall Notation \ref{notation:state-invariant-fock}, so that
$$
\|\pi_k(T) \| \geq \| \pi(T^{(n)}_{E_{kk}})(e^{(0)}_{kk}) \| = \| \frac{1}{P^{(n)}_{kk}}e^{(n)}_{kk} \| = \frac{1}{P^{(n)}_{kk}}
$$
On the other hand, $q(T^{(n)}_{E_{kk}}) = (z \mapsto z^m E_{kk})$ for $m\in \nn$ satisfying $n = rm$, so that 
$$
\| (ev_{\lambda} \circ q)(T) \| = \| ev_{\lambda}(z \mapsto z^m E_{kk}) \| = |\lambda^m| = 1
$$
So we see that $\|\pi_k(T) \| > \sup_{\lambda \in \mathbb{T}} \| (ev_{\lambda} \circ q)(T) \|$.

Next, fix $s \in \Omega$ with $k \neq s$. Since $T^*T \in \LL(\FF_{Arv(P)})$ sends $Arv(P)_m$ to $Arv(P)_m$, it is a finite-block diagonal operator, so we must have that $T^*T |_{\FF_{P,s}} = (T^{(n)}_{E_{kk}})^*(T^{(n)}_{E_{kk}}) |_{\FF_{P,s}}$ is also finite-block diagonal. Denote $I(k,s) = \{ \ m \in \nn \ | \ (k,s) \in E(P^m), \ m\geq 1 \ \}$, and note that since $T|_{\FF_{P,s}}(Arv(P)_{0,s}) = 0$, we have that
$$
\| \pi_{s}(T) \|^2 = \| \pi_{s}(T^*T) \| = \| T^*T |_{\FF_{P,s}} \| = \sup_{m\in \nn} \| T^*T |_{Arv(P)_{m,s}} \| = 
$$ 
$$
\sup_{m\in I(k,s)} \| (T^{(n)}_{E_{kk}})^*(T^{(n)}_{E_{kk}})(e^{(m)}_{ks}) \| = \sup_{m \in I(k,s)} \frac{P^{(m)}_{ks}}{P^{(n+m)}_{ks}} \| e^{(m)}_{ks} \| = \sup_{m \in I(k,s)} \frac{P^{(m)}_{ks}}{P^{(n+m)}_{ks}}
$$
By Theorem \ref{theorem:convergence-theorem-finite} we see that as $m \in I(k,s)$ goes to infinity, the fraction $\frac{P^{(m)}_{ks}}{P^{(n+m)}_{ks}}$ converges to the constant $\frac{\nu_{s} r}{\nu_{s} r} = 1$. 

Hence, if $\sup_{m \in I(k,s)} \frac{P^{(m)}_{ks}}{P^{(n+m)}_{ks}} \leq 1$, as $k$ is non-exclusive, we have that $P^{(n)}_{kk} < 1$ for large enough $n$, and so that $\| \pi_k(T) \| > \| \pi_{s}(T) \|$. 

On the other hand, if $\sup_{m \in I(k,s)} \frac{P^{(m)}_{ks}}{P^{(n+m)}_{ks}} > 1$, then the supremum above is in fact a maximum, and $s$ must be non-exclusive. By item (2) of Lemma \ref{lemma:non-exclusive-char} there is $n$ large enough (which we now prescribe) so that $0 < P^{(n)}_{kk} < 1$ and $P^{(n)}_{kk}P^{(m)}_{ks} < P^{(n+m)}_{ks}$ for all $m\in I(k,s)$. Hence, we see that $\frac{1}{P^{(n)}_{kk}} > \frac{P^{(m)}_{ks}}{P^{(n+m)}_{ks}}$ for all $m\in I(k,s)$ so that still we obtain $\| \pi_k(T) \| > \| \pi_{s}(T) \|$.

To conclude, we have shown that $\| \pi_k(T) \| > \max\{ \sup_{s \neq k } \{ \| \pi_{s}(T) \| \} , \sup_{\lambda \in \mathbb{T}} \|(ev_{\lambda} \circ q)(T) \| \} $ so that by \cite[Theorem 7.2]{arveson-non-commutative-choquet-II} we have that $\pi_k$ is a boundary representation.

(2): Suppose that $k$ is exclusive. By the formula for $T^{(n)}_{A}$, we see that $\pi_k(T^{(n)}_{A}) = \pi_k(W^{(n)}_{A})$. Indeed, this follows since any weights appearing in an application of $T^{(n)}_{A}$ to a $k$-th column of a matrix $B \in Arv(P)_m$ arise only from entries of the $k$-th columns of $P^n$, which are either $0$ or $1$ by assumption on $k$. 

We will use the above to show that $\pi_k$ is not strongly peaking anywhere by showing that it is not strongly peaking at any $\sum_{n=-N}^N[T_{ij}] \in M_s(\tensor(P)^* + \tensor(P))$ where each $T_{ij} \in \tensor(P)^* + \tensor(P)$ is of degree $n \in [-N, N]$ (which must then be either of the form $T^{(n)}_A$ or $T^{(n)*}_A$). We also denote by $W_{ij}$ the element (which is either of the form $W^{(n)}_{A}$ or $W^{(n)*}_A$ respectively) satisfying $\pi_k(T_{ij}) = \pi_k(W_{ij})$ above.

We note finally that there exists $m_0$ such that for all $m\geq m_0$ we have that $(U^m)^*W_{ij} U^m = W_{ij}$ for all $i,j$. We then have that
$$
\|\pi^{(s)}_k(\sum_{n=-N}^N[T_{ij}]) \| = \| \sum_{n=-N}^N[\pi_k(T_{ij})] \| = \| \sum_{n=-N}^N[\pi_k(W_{ij})] \| =
$$
$$
\| \sum_{n=-N}^N[ (U^m \otimes Id)^* \pi_k(W_{ij}) (U^m \otimes Id)] \| = \| \sum_{n=-N}^N[\pi_k(U^{m*}W_{ij}U^m) ] \|
$$
$$
\leq \| [U^{m*} (\sum_{n=-N}^N W_{ij}) U^m] \| \leq  \| [Q_{[m, \infty)} (\sum_{n=-N}^N W_{ij}) Q_{[m, \infty)}] \|
$$
So we see that by Proposition \ref{prop:complete-norm-cuntz}
$$
\|\pi^{(s)}_k(\sum_{n=-N}^N[T_{ij}]) \| \leq \lim_{m \rightarrow \infty} \| [(\sum_{n=-N}^N W_{ij}) Q_{[m, \infty)}] \| = \| [q(\sum_{n=-N}^N W_{ij})] \| = \| q^{(s)}(\sum_{n=-N}^N [T_{ij}]) \|
$$

Since $q^{(s)}(\sum_{n=-N}^N [T_{ij}]) \in C(\mathbb{T},M_d) \otimes M_s$, there exists $\lambda \in \mathbb{T}$ such that 
$$
\|\pi^{(s)}_k(\sum_{n=-N}^N[T_{ij}]) \| \leq \|q^{(s)}(\sum_{n=-N}^N [T_{ij}])\| = \| (ev_{\lambda} \circ q)^{(s)}(\sum_{n=-N}^N [T_{ij}])\|
$$
Since elements of the form $\sum_{n=-N}^N[T_{ij}]$ with $T_{ij}$ of degree $n$ are dense inside $M_s(\tensor(P)^* + \tensor(P))$, we see that for any $[V_{ij}] \in M_s(\tensor(P)^* + \tensor(P))$ we have $\|[\pi_k(V_{ij})]\| \leq \sup_{\lambda \in \mathbb{T}} \| [(ev_{\lambda} \circ q)(V_{ij}) \|$ so that $\pi_k$ cannot be strongly peaking. By \cite[Theorem 7.2]{arveson-non-commutative-choquet-II} we see that $\pi_k$ is not a boundary representation.
\end{proof}

\begin{remark}
It is clear that for exclusive $k\in \Omega$ the representation $\pi_k$ is not boundary by item (2) of Propositon \ref{prop:boundary-rep-exculsive}. Item (1) in Proposition \ref{prop:boundary-rep-exculsive} above provides a sufficient condition for $\pi_k$ to be boundary when $k$ is non-exclusive. We believe that this condition is not necessary, however we do not have examples to that effect.
\end{remark}

We next introduce a class of stochastic matrices for which we can completely identify the non-commutative Choquet boundary of $\tensor(P)$ inside $\toeplitz(P)$ in terms of the matrix $P$.

\begin{defi} \label{def:multiple-arrival}
Let $P$ be a finite $r$-periodic irreducible stochastic matrix over $\Omega$. We say that $P$ has the \emph{multiple-arrival} property if whenever $k, s \in \Omega$ are distinct non-exclusive states such that whenever $k$ leads to $s$ in $n$ steps, then there exists $k \neq k' \in \Omega$ such that $k'$ leads to $s$ in $n$ steps.
\end{defi}

\begin{cor} \label{cor:multiple-arrival-boundary}
Let $P$ be a finite irreducible stochastic matrix over $\Omega$, and $k\in \Omega$. If $P$ has the multiple-arrival property, then $\pi_k$ is a boundary representation if and only if $k$ is non-exclusive. Hence, the non-commutative Choquet boundary of $\tensor(P)$ inside $\toeplitz(P)$ is parameterized by a circle $\mathbb{T}$ of irreducible representations of dimension $d$, and irreducible representations $\pi_k$ of infinite dimension associated to non-exclusive states $k\in \Omega$.
\end{cor}

\begin{proof}
This follows directly since if $P$ has multiple-arrival, then the conditions of Proposition \ref{prop:boundary-rep-exculsive} item (1) are automatically satisfied for any non-exculsive $k\in \Omega$, and item (2) of Proposition \ref{prop:boundary-rep-exculsive} then gives the reverse implication.
\end{proof}

There is an easy class of examples which automatically has the multiple arrival property. Suppose that $P$ is an irreducible $r$-periodic stochastic matrix with cyclic decomposition $\Omega_0,..., \Omega_{r-1}$. Then we may write
$$
\left[\begin{smallmatrix}
       0 & P_0  & \cdots & 0 \\
       \vdots  & \ddots  & \ddots & \vdots  \\
       0 &  \cdots & 0 & P_{r-2} \\
       P_{r-1}  & \cdots & 0 & 0
     \end{smallmatrix} \right]
$$
for rectangle stochastic matrices $P_0,...,P_{r-1}$. If all entries of the matrices $P_0,...,P_{r-1}$ are non-zero, then $P$ is called \emph{fully-supported}, and has the multiple-arrival property.

Suppose $P$ is a finite irreducible stochastic matrix $P$ over $\Omega$ of size $d$. We next discuss $\cenv(\tensor(P))$ and its spectrum. Denote by $\Omega_b$ the set of states $k$ for which $\pi_k$ is a boundary representation, which is a subset of $\Omega - \Omega_e$. Since for all $k\in \Omega$ and $\lambda \in \mathbb{T}$ we have that $\Ker \pi_k \subset \JJ \subset \ker ev_{\lambda} \circ q$, and since the intersection of kernels of all boundary representations is the Shilov ideal, we must have that $\pi^{-1}(\oplus_{k \in  \Omega - \Omega_b}\KK(\FF_{P,k}))$ is the Shilov ideal of $\tensor(P)$ inside $\toeplitz(P)$, thought of as a subalgebra of $\pi^{-1}(\oplus_{k \in \Omega}\KK(\FF_{P,k})) = \JJ$.

We hence get the following short exact sequence
\begin{equation} \label{eq:cenv-short-exact}
0 \longrightarrow \oplus_{k \in \Omega_b}\KK(\FF_{P,k}) \longrightarrow \cenv(\tensor(P)) \longrightarrow C(\mathbb{T},M_d) \longrightarrow 0
\end{equation}
while we identify $q_e(\JJ(P)) \subset \cenv(\tensor(P))$ with $\oplus_{k \in \Omega_b}\KK(\FF_{P,k})$, where $q_e : \toeplitz(P) \rightarrow \cenv(\tensor(P))$ is the quotient map by the Shilov ideal.

If $\rho : \cenv(\tensor(P)) \rightarrow B(H)$ is a unital *-representation, it decomposes uniquely into a central direct sum of representations $\rho = \rho_{\JJ} \oplus \rho_{\cuntz}$, where $\rho_{\JJ}$ is the unique extension to $\cenv(\tensor(P))$ of the restriction of $\rho$ to $q_e(\JJ(P))$, and $\rho_{\cuntz}$ annihilates $q_e(\JJ(P))$. Hence, the spectrum of $\cenv(\tensor(P))$ decomposes into a disjoint union of the spectrum of $\oplus_{k \in \Omega_b}\KK(\FF_{P,k})$ and the spectrum of $C(\mathbb{T},M_d)$. That is, the spectrum of $\cenv(\tensor(P))$ is comprised of $|\Omega_b|$ irreducible representations of infinite dimension, and a torus $\mathbb{T}$ of irreducible representations of dimension $d$ that annihilate $q_e(\JJ(P))$.

\begin{theorem} \label{theorem:shilov-ideal-hyperrigid}
Suppose that $P$ is a finite irreducible matrix over $\Omega$. Then $\tensor(P)$ is hyperrigid in $\cenv(\tensor(P))$. Moreover, if $P$ has multiple-arrival, the Shilov ideal for $\tensor(P)$ inside $\toeplitz(P)$ is given by
$$
\mathcal{S}_P = \bigcap_{k \in \Omega - \Omega_e} \{ \ T \in \JJ(P) \ | \ \pi_k(T) = 0 \ \}
$$
and is *-isomorphic via $\pi$ to $\oplus_{k \in \Omega_e} \KK(\FF_{P,k})$
\end{theorem}

\begin{proof}
Let $\rho : \cenv(\tensor(P)) \rightarrow B(H)$ be a *-representation. By the above discussion, we may decompose it into a central direct sum of representations $\rho = \rho_{\JJ} \oplus \rho_{\cuntz}$, where $\rho_{\JJ}$ is the unique extension to $\cenv(\tensor(P))$ of the restriction of $\rho$ to $q_e(\JJ(P))$, and $\rho_{\cuntz}$ annihilates $q_e(\JJ(P))$.

By Proposition \ref{prop:annihi-auto-uep} we have that $\rho_{\cuntz} \circ q_e$ has u.e.p. when restricted to $\tensor(P)$, so that $\rho_{\cuntz}$ has u.e.p. when restricted to $q_e(\tensor(P))$ by invariance of maximal UCP maps. Next, since $\rho_{\JJ} \circ q_e = \oplus_{k \in \Omega_b} n_k \cdot \pi_k$ is a direct sum of *-representations, with certain multiplicities $n_k$, that have u.e.p. when restricted to $\tensor(P)$, by \cite[Theorem 4.4]{arveson-non-commutative-choquet-II} $\rho_{\JJ} \circ q_e$ has u.e.p. when restricted to $\tensor(P)$. Hence, again by invariance of maximal UCP maps, $\rho_{\JJ}$ has u.e.p. when restricted to $q_e(\tensor(P))$. By another application of \cite[Theorem 4.4]{arveson-non-commutative-choquet-II} we obtain that $\rho = \rho_{\JJ} \oplus \rho_{\cuntz}$ also has u.e.p. when restricted to $q_e(\tensor(P))$, so that $\tensor(P)$, which is completely isometric to $q_e(\tensor(P))$ via $q_e$, is hyperrigid within $\cenv(\tensor(P))$.

For the second part, by Corollary \ref{cor:multiple-arrival-boundary} we know that $\Omega_e = \Omega - \Omega_b$. Furthermore, by Proposition \ref{prop:annihi-auto-uep}, we have that $ev_{\lambda} \circ q$ is a boundary representation for $\tensor(P)$ for any $\lambda \in \mathbb{T}$, and since $\JJ(P) = \bigcap_{\lambda \in \mathbb{T}}\Ker(ev_{\lambda} \circ q)$, by the discussion preceding the theorem, the Shilov ideal must equal 
$$
\JJ(P) \cap \bigcap_{k\in \Omega_b} \Ker(\pi_k) = \bigcap_{k \in \Omega - \Omega_e} \{ \ T \in \JJ(P) \ | \ \pi_k(T) = 0 \ \}
$$
\end{proof}

We now give equivalent conditions that guarantee that the C*-envelope of $\tensor(P)$ is either the Toeplitz algebra, or the Cuntz-Pimsner algebra.

\begin{cor}
Let $P$ be a finite irreducible stochastic matrix of size $d$ with multiple-arrival. Then we have that $\cenv(\tensor(P)) \cong \toeplitz(P)$ if and only if all $k\in \Omega$ are non-exclusive. 

In particular, if $P$ aperiodic and of size $d \geq 2$ with multiple-arrival, we have that $\cenv(P) \cong \toeplitz(P)$.
\end{cor}

\begin{proof}
By Theorem \ref{theorem:shilov-ideal-hyperrigid} we see that if all $k\in \Omega$ are non-exclusive, then $\mathcal{S}_P = \{0\}$ and so $\toeplitz(P)$ is the C*-envelope of $\tensor(P)$.

Conversely, if $\toeplitz(P) \cong \cenv(\tensor(P))$, then $\cenv(\tensor(P))$ has $d$ irreducible representations of infinite dimension, which can only occur if $|\Omega_b| = d$. Since $P$ has multiple-arrival, we see by Corollary \ref{cor:multiple-arrival-boundary} that $\Omega_b = \Omega - \Omega_e$, and so all states $k\in \Omega$ are non-exclusive. 

For the second part, our assumptions guarantee that $\Omega = \Omega_0$ is the cyclic decomposition for $P$, and $|\Omega| > 1$, so by Lemma \ref{lemma:non-exclusive-char} all elements in $\Omega$ are non-exclusive. Since $P$ has multiple-arrival, by Corollary \ref{cor:multiple-arrival-boundary} we have that $\pi_k$ is a boundary representation for any $k\in \Omega$, and so $\mathcal{S}_P = \{0\}$ and $\cenv(\tensor(P)) \cong \toeplitz(P)$.
\end{proof}

\begin{cor} \label{cor:homomorphism-cuntz}
Let $P$ be a finite $r$-periodic irreducible stochastic matrix of size $d$, and let $\Omega_0,...,\Omega_{r-1}$ be a cyclic decomposition for $P$. The following are equivalent:
\begin{enumerate}
\item
All $k\in \Omega$ are exclusive.
\item
$|\Omega_{\ell}| = 1$ for all $\ell \in \mathbb{Z}_r$, or equivalently $r=d$.
\item
$P : \ell^{\infty}(\Omega) \rightarrow \ell^{\infty}(\Omega)$ is a *-homomorphism.
\item
$\cenv(\tensor(P)) \cong \cuntz(P)$

\end{enumerate}
\end{cor}

\begin{proof}
$(1) \Rightarrow (2)$: By item (1) of Lemma \ref{lemma:non-exclusive-char} we see that $|\Omega_{\ell}| = 1$ for all $\ell \in \mathbb{Z}_r$.

$(2) \Rightarrow (3)$: If all $\Omega_{\ell}$ are of size $1$, we see that the cyclic decomposition for $P$ yields that $P$ is in fact a permutation matrix of a single-cycle permutation, and is hence a homomorphism.

$(3) \Rightarrow (4)$: The Arveson-Stinespring construction of a subproduct system generally yields a product system when applied to a *-homomorphism. Hence, $Arv(P)$ is a product system, and its tensor algebra is the tensor algebra of a single correspondence, and by \cite[Proposition 2.8]{viselter2} this is also true for the Cuntz-Pimsner algebra in our case. By \cite[Theorem 3.7]{katsoulis-kribs-C-envelope} we have $\cenv(\tensor(P)) \cong \cuntz(P)$.

$(4) \Rightarrow (1)$: Assume towards contradiction that there is some $k\in \Omega$ that is non-exclusive. In this case, let $n$ be so that $0 < P^{(n)}_{kk} < 1$, and observe that $\| q(T^{(n)}_{E_{kk}}) \| = \| q(W^{(n)}_{E_{kk}}) \| = 1$, while for $E_{kk} \in Arv(P)_0$ we have
$$
\|T^{(n)}_{E_{kk}}\| \geq \|T^{(n)}_{E_{kk}}(E^{(0)}_{kk}) \| = \frac{1}{P^{(n)}_{kk}} \|E^{(n)}_{kk}\| = \frac{1}{P^{(n)}_{kk}}
$$
This means that $q : \toeplitz(P) \rightarrow \cuntz(P)$ is not isometric on $\tensor(P)^* + \tensor(P)$, and in particular, not completely isometric on $\tensor(P)^* + \tensor(P)$. By \cite[Theorem 7.2]{arveson-non-commutative-choquet-II}, there is a boundary representation for $\tensor(P)$ coming from an extension to $\toeplitz(P)$, of an element in the spectrum of $\JJ(P)$, which then must be equivalent to one of the $\pi_k$. This means that $\cenv(\tensor(P))$ has an irreducible representation of infinite dimension, which is impossible since $\cenv(\tensor(P)) \cong C(\mathbb{T},M_d)$ only has irreducible representations of dimension $d$.
\end{proof}

\begin{example} \label{ex:distinct}
We next give an example of $3 \times 3$ stochastic matrix for which $\cenv(\tensor(P))$, $\toeplitz(P)$ and $\cuntz(P)$ are pairwise non *-isomorphic. Let 
$$
P =  \begin{bmatrix}
       0 & 0  & 1 \\
       0  & 0  & 1 \\
       \frac{1}{2} & \frac{1}{2} & 0 
     \end{bmatrix}
$$
The matrix $P$ is fully-supported and we see that states $1$ and $2$ are non-exclusive, while $3$ is exclusive. Hence, $\Omega_b = \Omega - \Omega_e \subsetneq \Omega$. Therefore, the Shilov ideal $\mathcal{S}_P \cong \KK(H_1) \oplus \KK(H_2) \not \cong \oplus_{j\in \Omega}\KK(H_j) \cong \JJ(P)$. This yields a quotient $\cenv(\tensor(P))$ for which $\cenv(\tensor(P))$, $\toeplitz(P)$ and $\cuntz(P)$ are pairwise non *-isomorphic. 

\end{example}

Without the irreducibility assumption on $P$, it is easy to construct intermediary C*-envelopes from "extremal" C*-envelopes. Indeed, if for finite stochastic matrices $P$ and $Q$ of sizes at least $2$ we have that $\cenv(\tensor(P)) = \toeplitz(P)$, and $\cenv(\tensor(Q)) = \cuntz(Q)$, then $R=P\oplus Q$ is a finite stochastic matrix such that $\cenv(\tensor(R)) = \cenv(\tensor(P)) \oplus \cenv(\tensor(Q)) = \toeplitz(P) \oplus \cuntz(Q)$, and one can similarly use representation theory to show that $\cenv(\tensor(R))$ is non *-isomorphic to $\toeplitz(R)$ nor $\cuntz(R)$.

However, when $P$ is irreducible, the subproduct system $Arv(P)$ associated to it, cannot have any non-trivial reducing projections in the sense of Remark \ref{rem:sps-reducing-proj}, so we have irreducibility both in a dynamical sense and in a sense of its subproduct system. The above example then shows that even under this minimality / irreducibility assumptions on a dynamical object which is equivalent to this irreducibility assumptions on the subproduct system above, up to *-isomorophism the C*-envelope may be distinct from both the Cuntz-Pimsner algebra, and the Toeplitz algebra.

\section{K-Theory} \label{Sec:K-theory}

In this section we compute the K-theory of $\cenv(\tensor(P))$. Recall Notation \ref{notation:state-invariant-fock}. Let $\Omega_b$ be the collection of $k \in \Omega$ for which $\pi_k$ is a boundary representation. We note immediately that we identify $\toeplitz(P)$ with its image under $\pi : \toeplitz(P) \rightarrow B(\oplus_{k\in \Omega} \FF_{P,k})$, where $\FF_{P,k}$ are the invariant subspaces of $\pi$ and $\pi_k : \toeplitz(P) \rightarrow B(\FF_{P,k})$ the irreducible, pairwise non-unitarily equivalent representations given by restriction $\pi_k(T) = T|_{\FF_{P,k}}$, for each $k\in \Omega$. Recall the short exact sequence from \eqref{eq:cenv-short-exact}.

We know from \cite{RordamBook} that $K_0$ and $K_1$ are additive functors, and that for any $k\in \Omega$ we have $K_1(\KK(\FF_{P,k})) = \{0\}$, and $K_0(C(\mathbb{T},M_d)) \cong K_0(\KK(\FF_{P,k})) \cong K_1(C(\mathbb{T},M_d)) \cong \mathbb{Z}$. Hence, the six-term exact sequence of K-theory induced from the exact sequence \eqref{eq:cenv-short-exact} yields
\begin{equation} \label{eq:six-term-exact}
 \begin{array}{ccccl}
  0 & \longrightarrow & K_1(C_{env}^*(\tensor(P))) & \longrightarrow & \mathbb{Z} \\
  \uparrow & \ & \ & \ & \downarrow \delta_1 \\
  \mathbb{Z} & \longleftarrow & K_0(C_{env}^*(\tensor(P))) & \longleftarrow & \mathbb{Z}^{|\Omega_b|}
\end{array}
\end{equation}
Our goal in this section is to compute the index map $\delta_1 : K_1(C(\mathbb{T},M_d)) \rightarrow K_0(\oplus_{k\in \Omega_b} \KK(\FF_{P,k}))$, which will then enable the computation of the $K_0$ and $K_1$ groups for $\cenv(\tensor(P))$. It will suffice to compute the value of $\delta_1$ on a generator of $K_1(C(\mathbb{T},M_d)) \cong \mathbb{Z}$, and in our computations we will work with the unitary element $w : = z \mapsto diag(z,1,...,1) \in C(\mathbb{T},M_d)$, as $[w]_1$ is a generator for $K_1(C(\mathbb{T},M_d)) \cong \mathbb{Z}$.

\begin{lemma} \label{lemma:V-partial-isom}
Let $P$ be an $r$-periodic irreducible stochastic matrix over $\Omega$ of size $d$, with properly enumerated cyclic decomposition $\Omega_0,...,\Omega_{r-1}$ such that $1\in \Omega_0$ is the first element, and let $(U, (S_{ij})_{i,j\in \Omega})$ be its associated standard family. 
\begin{enumerate}
\item
For all $i\in \Omega$ we have that $(S_{ii} = p_i)_{i\in \Omega}$ is a family of pairwise orthogonal projections that commute with $U$.
\item
The element $w: = z \mapsto diag(z,1,...,1) \in C(\mathbb{T},M_d)$ lifts to a partial isometry $V := US_{11} + S_{22} + ... S_{dd}$ inside $\toeplitz(P)$.
\end{enumerate}
\end{lemma}

\begin{proof}
(1): By definition, for any $m \in \nn$ and $E_{jk} \in Arv(P)_m$ we have that $S_{ii}(E_{jk}) = Gr(P^m) * (E_{ii} \cdot E_{jk}) = \delta_{i,j} E_{jk} = p_i(E_{jk})$ so that $S_{ii} = p_i$. Next, note that
$$
U S_{ii}(E_{jk}) = Gr(P^{m+r})*(\delta_{ij} E_{jk}) = \delta_{ij} U(E_{jk}) = S_{ii}U(E_{jk})
$$
So that $U$ and $S_{ii}$ commute on the dense subset of $\FF_{Arv(P)}$, and hence commute.

(2): It is clear that $w$ lifts to $V$ inside $\toeplitz(P)$ since under the identification $C(\mathbb{T}) \otimes M_d \cong C(\mathbb{T};M_d)$, the element $\overline{V}$ in the quotient is identified with $w$. Hence, we need only verify that $V$ is a partial isometry. Indeed, since by item (1), $U$ commutes with $S_{11}$, and since $U$ is a partial isometry, we have that
$$
VV^*V = S_{11}UU^*US_{11} + S_{22} + ... S_{dd} = V
$$
so that $V$ is also a partial isometry.
\end{proof}

Let $v$ be the image of $V$ under the C*-envelope quotient map $q_e: \toeplitz(P) \rightarrow \cenv(\tensor(P))$. By Lemma \ref{lemma:V-partial-isom} we know that $V$ is a partial isometry that lifts $w$, and hence $v$ is a partial isometry that lifts $w$. By item (ii) of \cite[Proposition 9.2.5]{RordamBook} we have that $1-v^*v$ and $1-vv^*$ are projections in $\oplus_{k \in \Omega_b} \KK(H_k) \cong q_e(\JJ(P))$ with
$$
\delta_1([w]_1) = [1 - v^*v]_0 - [1 - vv^*]_0
$$
But due to the identification $K_0(\oplus_{k\in \Omega_b} \KK(\FF_{P,k})) \cong \oplus_{k \in \Omega_b} K_0 (\KK(\FF_{P,k}))$, we obtain that
$$
\delta_1([w]_1) = [1-v^*v]_0 - [1 - vv^*]_0 = \Big( [(1-V^*V) |_{\FF_{P,k}}] - [(1 - VV^*) |_{\FF_{P,k}}]  \Big)_{k\in \Omega_b} =
$$
$$
\Big( \dim \Ker (V| _{\FF_{P,k}}) - \dim \Ker (V^*| _{\FF_{P,k}})  \Big)_{k\in \Omega_b} = \big( \ind(V |_{\FF_{P,k}}) \big)_{k\in \Omega_b}
$$
Where we are then left with computing the Fredholm indices of $V|_{\FF_{P,k}}$ for $k \in \Omega_b$.

\begin{prop} \label{prop:pi-index}
Let $P$ be an $r$-periodic irreducible stochastic matrix over $\Omega = \{1,...,d\}$. Suppose that $\Omega_0,...,\Omega_{r-1}$ is a properly enumerated cyclic decomposition such that $s \in \Omega$ is its first element, and let $(U, (S_{ij})_{i,j\in \Omega})$ be its associated standard family. Let $V_s: = S_{11} + ... + S_{s-1,s-1} + US_{ss} + S_{s+1,s+1} + ... + S_{dd}$. Then for every $k\in \Omega$ we have that $\ind(V_s|_{\FF_{P,k}}) = -1$.
\end{prop}

\begin{proof}
Up to conjugating $P$ with a permutation matrix, we may assume that $s=1$ is the first element.
For each state $k\in \Omega$, let $\ell = \sigma(k) = \sigma(k) - \sigma(1)$, and denote by
$$
b_n = \begin{cases} 
      1 & \text{: } P_{1k}^{(nr+\ell)} > 0 \\
      0 & \text{: } P_{1k}^{(nr+\ell)} = 0
   \end{cases}
$$
where $P^{(0)}_{1k} = 1$ if $k=1$, and is $0$ otherwise. Recall Notation \ref{notation:state-invariant-fock}. Since $\FF_{P,k} = \oplus_{n=0}^{\infty}Arv(P)_{n,k}$, and since $V$ shifts only the first rows of the matrix $A$ in an element $A \otimes e_k \in Arv(P)_{n,k}$, we have for all $n\in \nn$ that
$$
\dim \Ker V|_{Arv(P)_{n,k}} = b_n - b_{n+1}b_n 
$$
and for all $n\geq 1$ that
$$
\dim\Ker V^*|_{Arv(P)_{n,k}} = b_{n+1} - b_{n+1}b_n
$$
due to the support of elements in $Arv(P)_{n,k}$. Note also that for $n=0$, and we get
$$
\dim \Ker V^*|_{Arv(P)_{0,k}} = \begin{cases} 
0 & \text{: } k \neq 1 \\
1 & \text{: } k = 1
\end{cases}
$$
Hence, if we sum up dimensions, we obtain that
\begin{equation*}
\dim \Ker V|_{\FF_{P,k}} = \sum_{n=0}^{\infty} b_n - b_{n+1}b_n
\end{equation*}
and
\begin{equation*}
\dim \Ker V^*|_{\FF_{P,k}} = \begin{cases} 
\sum_{n=0}^{\infty} b_{n+1} - b_{n+1}b_n & \text{: } k \neq 1 \\
1+ \sum_{n=0}^{\infty} b_{n+1} - b_{n+1}b_n & \text{: } k = 1
\end{cases}
\end{equation*}
Thus, 
\begin{equation*}
\ind(V|_{\FF_{P,k}}) = \dim \Ker V|_{\FF_{P,k}} - \dim \Ker V^*|_{\FF_{P,k}}
\end{equation*}

\begin{equation*}
=\begin{cases} 
\sum_{n=0}^{\infty} b_n - b_{n+1} & \text{ if } k \neq 1 \\
-1 + \sum_{n=0}^{\infty} b_n - b_{n+1} & \text{ if } k = 1
\end{cases}
\ \ = \ \ 
\begin{cases} 
0 - 1 & \text{ if } k \neq 1 \\
-1 + 1 - 1 & \text{ if } k = 1
\end{cases}
\end{equation*}
Hence, we see that in any case, $\ind(V|_{\FF_{P,k}}) = -1$, as required.
\end{proof}

\begin{cor} \label{cor:delta-computation}
Let $P$ be an irreducible stochastic matrix over finite $\Omega$. Then the index map $\delta_1 : \mathbb{Z} \rightarrow \mathbb{Z}^{|\Omega_b|}$ is given by $\delta_1(n) = -(n,...,n)$
\end{cor}
We then obtain the K-theory of $\cenv(\tensor(P))$ in terms of $|\Omega_b|$.

\begin{theorem} \label{theorem:k-theory-computation}
Let $P$ be a finite irreducible stochastic matrix over $\Omega$. Then
\begin{enumerate}
\item
If $P$ has a non-exclusive state then
$$
K_0(\cenv(\tensor(P))) \cong
\mathbb{Z}^{|\Omega_b|} \ \text{ and } \ 
K_1(\cenv(\tensor(P))) \cong
\{0\}
$$
\item
If all states of $P$ are exclusive then
$$
K_0(\cenv(\tensor(P))) \cong
\mathbb{Z} \ \text{ and } \ 
K_1(\cenv(\tensor(P))) \cong
\mathbb{Z}
$$
\end{enumerate}

\end{theorem}

\begin{proof}
If all states of $P$ are exclusive, then by Corollary \ref{cor:homomorphism-cuntz} we have that $\cenv(\tensor(P)) \cong C(\mathbb{T},M_d)$ so that the $K_0$ and $K_1$ groups of $\cenv(\tensor(P))$ must both be $\mathbb{Z}$.

Next, if $P$ has a non-exclusive state, since $\delta_1$ is injective, by exactness at $K_1(C(\mathbb{T},M_d))$ in the six-term exact sequence \eqref{eq:six-term-exact}, we see that $K_1(\cenv(\tensor(P))) = \{0\}$.

Since $\delta_1(1) = (-1,...,-1)$, we see that the six-term exact sequence \eqref{eq:six-term-exact} can be reduced to the single exact sequence
$$
0 \leftarrow \mathbb{Z} \leftarrow K_0(\cenv(\tensor(P))) \leftarrow \mathbb{Z}^{|\Omega_b|} / Sp_{\mathbb{Z}} \big( (-1,...,-1) \big) \leftarrow 0
$$
Since $\mathbb{Z}^{|\Omega_b|} / Sp_{\mathbb{Z}} \big( (-1,...,-1) \big) \cong \mathbb{Z}^{|\Omega_b| -1}$, we see that $K_0(\cenv(\tensor(P))) \cong \mathbb{Z}^{|\Omega_b|}$, and the proof is complete.
\end{proof}

\begin{cor}
Let $P$ be a finite irreducible stochastic matrix over $\Omega$. Then conditions $(1)$ through $(4)$ of Corollary \ref{cor:homomorphism-cuntz} are all equivalent to $K_1(\cenv(\tensor(P))) \cong \mathbb{Z}$.
\end{cor}

\section{Classification of C*-envelope} \label{Sec:Classification}

We are now in a position to apply the theory in the previous sections to obtain classification results up to *-isomorphism and stable isomorphisms of C*-envelopes arising from finite irreducible stochastic matrices.

For every finite irreducible stochastic matrix $P$ over $\Omega^P$, which has at least one non-exclusive state, let $\Omega_b^P$ be the (non-empty) set of indices $k\in \Omega^P$ such that $\pi_k: \toeplitz(P) \rightarrow B(\FF_{P,k})$ is a boundary representation for $\tensor(P)$. We note that $\cenv(\tensor(P))$ is Type I (equivalently GCR), being an extension of a CCR algebra by a CCR algebra. Thus, we may identify an irreducible representation with its kernel when discussing elements of the primitive ideal spectrum of $\cenv(\tensor(P))$.

The analysis done around the exact sequence \eqref{eq:cenv-short-exact} shows that the spectrum of $\cenv(\tensor(P))$ as a set is comprised of $|\Omega^P_b|$ irreducible representations of infinite dimensions induced from $\pi_k$, which we still denote by $\pi_k : \cenv(\tensor(P)) \rightarrow B(\FF_{P,k})$ for $k\in \Omega^P_b$, and a torus $\mathbb{T}$ of irreducible representations of dimension $|\Omega^P|$ given by $ev_{\lambda} \circ q$ for every $\lambda \in \mathbb{T}$, where $q : \cenv(\tensor(P)) \rightarrow C(\mathbb{T},M_{|\Omega^P|})$ is the quotient map. Moreover, we have the exact sequence
\begin{equation*} \label{eq:weak-equiv-PQ}
0 \longrightarrow \oplus_{k \in \Omega^P_b}\KK(\FF_{P,k}) \overset{\iota}{\longrightarrow} \cenv(\tensor(P)) \overset{q}{\longrightarrow} C(\mathbb{T},M_{|\Omega_P|}) \longrightarrow 0
\end{equation*} 

Since $\Ker \pi_k \subseteq \Ker (ev_{\lambda} \circ q)$ for every $k\in \Omega^P_b$ and every $\lambda \in \mathbb{T}$, and $\Ker (ev_{\lambda} \circ q)$ is not a subset of any $\Ker (ev_{\lambda'} \circ q)$ for $\lambda' \neq \lambda \in \mathbb{T}$, we see that for every $\lambda \in \mathbb{T}$, each $\Ker (ev_{\lambda} \circ q)$ is a maximal element in the lattice $\Prim(\cenv(\tensor(P)))$.

\begin{notation}
For a finite irreducible stochastic matrix $P$, we denote from now on $K_P: = \oplus_{k \in \Omega^P_b}\KK(\FF_{P,k})$, $B_P:= C(\mathbb{T},M_{|\Omega_P|})$ and $A_P:=\cenv(\tensor(P))$.
\end{notation}

Let $P$ and $Q$ be irreducible stochastic matrices over finite sets $\Omega^P$ and $\Omega^Q$ respectively. Then we have the following exact sequences
\begin{equation} \label{eq:exact-seq-PQ}
0 \rightarrow K_P \rightarrow A_P \rightarrow B_P \rightarrow 0
\ \ \text{and} \ \ 
0 \rightarrow K_Q \rightarrow A_Q \rightarrow B_Q \rightarrow 0 
\end{equation}
with Busby invariants $\eta_P$ and $\eta_Q$, and the stabilized exact sequences
\begin{equation} \label{eq:exact-seq-PQ-stab}
0 \rightarrow K_P \otimes \KK \rightarrow A_P \otimes \KK \rightarrow B_P \otimes \KK \rightarrow 0
\ \ \text{and} \ \ 
0 \rightarrow K_Q \otimes \KK \rightarrow A_Q \otimes \KK \rightarrow B_Q \otimes \KK \rightarrow 0 
\end{equation}
with Busby invariants $\eta_P^{(\infty)}$ and $\eta_Q^{(\infty)}$ given by $\eta_P^{(\infty)}([T_{ij}]) = [\eta_P(T_{ij})]$ for $[T_{ij}] \in B_P \otimes \KK$.

For $k\in \Omega_b$ denote by $\rho_k: \oplus_{\ell \in \Omega_b^P} B(\FF_{P,\ell}) \rightarrow B(\FF_{P,k})$ the restriction map, which then promotes to a restriction $\widetilde{\rho}_k : \oplus_{\ell \in \Omega_b^P} \QQ(\FF_{P,\ell}) \rightarrow \QQ(\FF_{P,k})$.

\begin{prop} \label{prop:C-env-char-ext}
Let $P$ and $Q$ be finite irreducible stochastic matrices over $\Omega^P$ and $\Omega^Q$ respectively. 

\begin{enumerate}
\item

$\cenv(\tensor(P))$ and $\cenv(\tensor(Q))$ are *-isomorphic if and only if there exists a *-isomorphism $\beta : C(\mathbb{T},M_{|\Omega_P|}) \rightarrow C(\mathbb{T},M_{|\Omega_Q|})$ and a bijection $\tau : \Omega_b^P \rightarrow \Omega_b^Q$ such that for all $k\in \Omega_b^P$ the extensions $\widetilde{\rho}_k \eta_P$ and $\widetilde{\rho}_{\tau(k)} \eta_Q \beta$ are strongly equivalent. 

\item
$\cenv(\tensor(P))$ and $\cenv(\tensor(Q))$ are stably isomorphic if and only if there exists a *-isomorphism $\beta : C(\mathbb{T},M_{|\Omega_P|}) \otimes \KK \rightarrow C(\mathbb{T},M_{|\Omega_Q|}) \otimes \KK$ and a bijection $\tau : \Omega_b^P \rightarrow \Omega_b^Q$ such that for all $k\in \Omega_b^P$ the extensions $\widetilde{\rho}_k^{(\infty)} \eta^{(\infty)}_P$ and $\widetilde{\rho}_{\tau(k)}^{(\infty)} \eta^{(\infty)}_Q \beta$ are weakly equivalent. 

\end{enumerate}
\end{prop}

\begin{proof}
We first show (1). Suppose that $\alpha: A_P \rightarrow A_Q$ is a *-isomorphism. Let $\alpha_* : \Prim(A_P) \rightarrow \Prim(A_Q)$ be the induced lattice isomorphisms between the spectra. Since $\alpha_*$ must send maximal elements to maximal elements, we see that for $\lambda \in \mathbb{T}$ we have that $\alpha_*$ sends $\Ker (ev^P_{\lambda} \circ q)$ to $ \Ker (ev^Q_{\lambda'} \circ q) $ for some $\lambda' \in \mathbb{T}$ in bijection. In particular, since $K_P = \cap_{\lambda \in \mathbb{T}} \Ker (ev^P_{\lambda} \circ q)$ and $K_Q = \cap_{\lambda \in \mathbb{T}} \Ker (ev^Q_{\lambda} \circ q)$, we see that $\alpha(K_P) = K_Q$. Hence $\cenv(\tensor(P))$ and $\cenv(\tensor(Q))$ are *-isomorphic if and only if the exact sequences of \eqref{eq:exact-seq-PQ} are isomorphic, which happens if and only if the restriction $\kappa : = \alpha|_{K_P} : K_P \rightarrow K_Q$ and the induced map $\beta$ satisfy $\widetilde{\kappa}\eta_P = \eta_Q \beta$, where $\beta : B_P \rightarrow B_Q$ is the induced *-isomorphism from $\alpha$ between the quotients by $K_P$ and by $K_Q$. 

So suppose $\widetilde{\kappa}\eta_P = \eta_Q \beta$ for $\kappa$ and $\beta$ as above. Since $\kappa: \oplus_{k\in \Omega_b^P} \KK(\FF_{P,k}) \rightarrow \oplus_{k\in \Omega_b^Q} \KK(\FF_{Q,k})$, there is a bijection $\tau : \Omega_b^P \rightarrow \Omega_b^Q$ and a unitaries $U_k : \FF_{P,k} \rightarrow \FF_{Q,\tau(k)}$ such that $\kappa|_{\KK(\FF_{P,k})} = Ad_{U_k} : \KK(\FF_{P,k}) \rightarrow \KK(\FF_{Q,\tau(k)})$, so that 
$$
\widetilde{\rho}_{\tau(k)}\eta_Q \beta = \widetilde{\rho}_{\tau(k)} \widetilde{\kappa} \eta_P = \widetilde{Ad_{U_k}} \widetilde{\rho}_k \eta_P
$$
For the converse, if $U_k$ are unitaries implementing the strong conjugacy between $\widetilde{\rho}_{\tau(k)}\eta_Q \beta$ and $\widetilde{Ad_{U_k}} \widetilde{\rho}_k \eta_P$, by setting $\kappa = \oplus_{k\in \Omega_P^b} Ad_{U_k} : \oplus_{k\in \Omega_b^P} \KK(\FF_{P,k}) \rightarrow \oplus_{k\in \Omega_b^Q} \KK(\FF_{Q,k})$, we have that
$$
\widetilde{\kappa}\eta_P = \oplus_{k\in \Omega_b^P} Ad_{U_k} \widetilde{\rho}_k \eta_P = \oplus_{k\in \Omega_b^Q} \widetilde{\rho}_{\tau(k)} \eta_Q \beta = \eta_Q \beta
$$

Next, we show (2). Since stabilizing an algebra does not change its primitive ideal spectrum, the same argument as used in (1) shows that $\cenv(\tensor(P))$ and $\cenv(\tensor(Q))$ are stably isomorphic if and only if the exact sequences in \eqref{eq:exact-seq-PQ-stab} are isomorphic, which happens if and only if there are *-isomorphisms $\kappa: K_P \otimes \KK \rightarrow K_Q \otimes \KK$ and $\beta : B_P \otimes \KK \rightarrow B_Q \otimes \KK$ such that $\widetilde{\kappa} \eta_P^{(\infty)} = \eta^{(\infty)}_Q \beta$. Then a similar argument to the one used for item (1) shows that this happens if and only if there is a bijection $\tau : \Omega_b^P \rightarrow \Omega_b^Q$ such that for all $k\in \Omega_b^P$ the extensions $\widetilde{\rho}_k^{(\infty)} \eta_P^{(\infty)}$ and $\widetilde{\rho}_{\tau(k)}^{(\infty)} \eta_Q^{(\infty)} \beta$ are \emph{strongly} equivalent. Since these are non-unital extensions, this happens if and only if they are weakly equivalent.
\end{proof}

For an irreducible finite stochastic matrix $P$ over $\Omega^P$ with period $r_P$, and $k\in \Omega^P$. Let $\Omega_0,...,\Omega_{r_P-1}$ be a cyclic decomposition for $P$. Then there exists $m_0$, such that for all $m \geq m_0$ we have 
$$
|\Omega_{\sigma(k) -m}| = | \{ \ i\in \Omega_{\sigma(k) -m} \ | \ P^{(m)}_{ik} > 0 \ \}|
$$
Indeed, fix $0 \leq \ell \leq r_P-1$. By item (2) of Theorem \ref{theorem:cyclic-graph-decomposition} there is $n^{(\ell)}_0$ such that for all $n\geq n^{(\ell)}_0$ we have that $P^{(nr_P + \ell)}_{ij}>0$ for $i,j\in \Omega^P$ with $\sigma(i) - \sigma(j) = \ell$. Hence, if we fix $j=k$, we see that 
$$
|\Omega_{\sigma(k) -(nr_P + \ell)}| = | \{ \ i\in \Omega_{\sigma(k) -(nr_P + \ell)} \ | \ P^{(nr_P + \ell)}_{ik} > 0 \ \}|
$$
Then simply take $m_0 = \max_{\ell}\{ n_0^{(\ell)}r_P+\ell \}$ to obtain the desired claim above.

\begin{defi} \label{def:column-nullity}
Let $P$ be an $r$-periodic finite irreducible stochastic matrix over $\Omega$ of size $d$, and $k\in \Omega$. Let $\Omega_0,..., \Omega_{r-1}$ be a cyclic decomposition for $P$, so that $\sigma(k)$ is the unique index such that $k \in \Omega_{\sigma(k)}$. We define the \emph{$k$-th column nullity} of $P$ to be
$$
\mathcal{N}_P(k) = \sum_{m=1}^{\infty} | \{ \ i\in \Omega_{\sigma(k) -m} \ | \ P^{(m)}_{ik} = 0 \ \}|
$$
where $\sigma(k) - m$ is taken as an element in the cyclic group $\mathbb{Z}_r$ of order $r$. We say that $k\in \Omega$ is a fully supported column if $\mathcal{N}_P(k) = 0$.
\end{defi}

Put in other words, the column nullity of a state $k \in \Omega$ is the number of zeros in all $k$-th columns of iterations of $P$, that lie in the support of a cyclic decomposition for $P$.

The above infinite sum is in fact always finite by the discussion preceding Definition \ref{def:column-nullity} and is hence convergent.

For a finite irreducible stochastic matrix $P$, we find the element in $\Ext_s(C(\mathbb{T}) \otimes M_d)$ representing each extension $\eta_{P,k}:= \widetilde{\rho}_k \eta_P$, for each $k\in \Omega_b^P$, appearing in Proposition \ref{prop:C-env-char-ext}. Note that the exact sequence corresponding to the extension $\eta_{P,k}$ is
$$
0 \rightarrow \KK(\FF_{P,k}) \rightarrow \pi_k(\toeplitz(P)) \rightarrow C(\mathbb{T}, M_{|\Omega^P|}) \rightarrow 0
$$
Recall the computation of $\Ext_s(C(\mathbb{T}) \otimes M_d)$ and $\Ext_w(C(\mathbb{T}) \otimes M_d)$ preceding Proposition \ref{prop:automorphism-invert}.

\begin{prop} \label{prop:defect-computation}
Let $P$ be a finite irreducible stochastic matrix over $\Omega^P$ with period $r_P$, and let $\Omega_0,...,\Omega_{r_P -1}$ be a properly enumerated cyclic decomposition for $P$. Then for each $k \in \Omega_b^P$, there exists $n_0$ large enough so that for all $n\geq n_0$ we have that $[j_*\eta_{P,k}]_s$ is identified with $0 \leq s < |\Omega^P|$ given by
$$
s \equiv \sum_{m=0}^{nr_P-1}|\{ \ i \in \Omega_{\sigma(k) - m} \ | \ P_{ik}^{(m)} > 0 \} | \mod |\Omega^P|
$$
and $[\iota_*\eta_{P,k}]_s$ is identified with $-|\Omega^P|$. In particular, $[\eta_{P,k}]_w = -1$.
\end{prop}

\begin{proof}
To compute the class of $[j_*\eta_{P,k}]$, we apply the algorithm in Example \ref{ex:matrix-extension} to $j_*\eta_{P,k}$. Let $\{\overline{S_{ij}}\}$ be the system of matrix units for $C(\mathbb{T},M_{|\Omega^P|})$ associated to a properly enumerated cyclic decomposition $\Omega_0,...,\Omega_{r_P-1}$, and let $1\in \Omega$ be the first element in this enumeration. There then exists $m_0$ such that for all $m \geq m_0$ we have $|\Omega_{\sigma(k) - m}| = | \{ \ i\in \Omega_{\sigma(k) -m} \ | \ P^{(m)}_{ik} > 0 \ \}|$. We abuse notation for sake of brevity and write $T$ instead of $\pi_k(T) = T|_{\FF_{P,k}}$ for $T \in \toeplitz(P)$.

Lift each $\overline{S_{ii}}$ to $p_i \cdot Q_{[nr_P,\infty)} \in \pi_k(\toeplitz(P))$. Then we may lift each $\overline{S_{1j}}$ to $S_{1j}Q_{[nr_P,\infty)} \in \pi_k(\toeplitz(P))$. Hence, we get that $e_{ij}:= Q_{[nr_P,\infty)} S_{1j}^*S_{1i}Q_{[nr_P,\infty)}$, so that for all $n\in \mathbb{N}$ with $nr_P \geq m_0$,
$$
p = \sum_{i=1}^{|\Omega^P|} Q_{[nr_P,\infty)} S_{1i}^*S_{1i}Q_{[nr_P,\infty)}
$$

Denote by $b^{(m)}_{ik}$ the indicator, which is $1$ if and only if $P_{ik}^{(m)}>0$ and $0$ otherwise. Then, the dimension of the cokernel of $p$ is congruent mod $|\Omega^P|$ to
$$
\sum_{i=1}^{|\Omega^P|} \sum_{m=0}^{nr_P-1} b^{(m)}_{ik} = \sum_{m=0}^{nr_P-1}| \{ \ i\in \Omega_{\sigma(k) -m} \ | \ P^{(m)}_{ik} > 0 \ \}|
$$
so we may take $n_0 = \lceil \frac{m_0}{r_P} \rceil$.

As for $\iota_* \eta_{P,k}$, a lift for $z\otimes I \in C(\mathbb{T}) \otimes I$ can be taken to be $U^P$, where $U^P$ is the unitary associated to the properly enumerated cyclic decomposition $\Omega_0,...,\Omega_{r_P-1}$ (restricted to $\FF_{P,k}$). Then by Proposition \ref{prop:pi-index} and the notation there, $\ind(U^P) = \ind( \Pi_{i\in \Omega}V_i) = - |\Omega^P|$. Finally, recall that the image $[\iota_* \eta_{P,k}]_s \in |\Omega^P| \cdot \mathbb{Z} $ is identified with $[\eta_{P,k}]_w \in \mathbb{Z}$ up to dividing by $|\Omega^P|$, so that $[\eta_{P,k}]_w = -1$.
\end{proof}

We now reach the two main results of this paper, which classify stable isomorphism and *-isomorphism of C*-envelopes in terms of the underlying stochastic matrices and boundary representations supported on different copies of compact operator subalgebras.

\begin{theorem} \label{thm:stable-iso}
Let $P$ and $Q$ be finite irreducible stochastic matrices over $\Omega^P$ and $\Omega^Q$ respectively. Then $|\Omega_b^P| = |\Omega_b^Q|$ if and only if $\cenv(\tensor(P))$ and $\cenv(\tensor(Q))$ are stably isomorphic.
\end{theorem}

\begin{proof}
If $\cenv(\tensor(P))$ and $\cenv(\tensor(Q))$ are stably isomorphic, since $K_0$ and $K_1$ are stable functors, we must have that $|\Omega_b^P| = |\Omega_b^Q|$ by Theorem \ref{theorem:k-theory-computation}.

For the converse, suppose $\Omega_b:= \Omega_b^P = \Omega_b^Q$.
For $k\in \Omega_b$, denote by $\eta_{P,k}^{(|\Omega^Q|)}$ and $\eta_{Q,k}^{(|\Omega^P|)}$ the ampliations of these extensions to $C(\mathbb{T}) \otimes M_{|\Omega_P|} \otimes M_{|\Omega_Q|}$. By Proposition \ref{prop:defect-computation} we then have that $\eta_{P,k}^{(|\Omega^Q|)}$ and $\eta_{Q,k}^{(|\Omega^P|)}$ are weakly unitarily equivalent. Hence, $\eta_{P,k}^{(\infty)}$ and $\eta_{Q,k}^{(\infty)}$ are also weakly equivalent, so that by item (2) of Proposition \ref{prop:C-env-char-ext} (with $\beta = Id$) we have that $\cenv(\tensor(P))$ and $\cenv(\tensor(Q))$ are stably isomorphic.
\end{proof}

\begin{theorem} \label{thm:*-iso}
Let $P$ and $Q$ be finite irreducible stochastic matrices over $\Omega^P$ and $\Omega^Q$ respectively. Then $\cenv(\tensor(P))$ and $\cenv(\tensor(Q))$ are *-isomorphic if and only if $d:=|\Omega^P| = |\Omega^Q|$ and there is a bijection $\tau : \Omega_b^P \rightarrow \Omega_b^Q$ such that for all $k\in \Omega_b^P$ we have $\mathcal{N}_P(k) \equiv \mathcal{N}_Q(\tau(k)) \mod d$.
\end{theorem}

\begin{proof}
Suppose $\cenv(\tensor(P))$ and $\cenv(\tensor(Q))$ are *-isomorphic. By item (1) of Proposition \ref{prop:C-env-char-ext} there is an *-isomorphism $\beta \in C(\mathbb{T}, M_{|\Omega^P|}) \rightarrow C(\mathbb{T},M_{|\Omega^Q|})$ (so that $d := |\Omega^P| = |\Omega^Q|$) and a bijection $\tau : \Omega_b^P \rightarrow \Omega_b^Q$ such that $\eta_{P,k}$ and $\eta_{Q,\tau(k)} \beta$ are strongly equivalent. By Proposition \ref{prop:automorphism-invert} $\beta_s$ is the identity on the second coordinate of $\Ext_s(C(\mathbb{T}) \otimes M_d) \cong d\mathbb{Z} \times \mathbb{Z}_d$. Hence, we see that $[j_*\eta_{P,k}] = [j_* \eta_{Q,\tau(k)}]$, so that $k\in \Omega_b^P$ and $\mathcal{N}_P(k) \equiv \mathcal{N}_Q(\tau(k)) \mod d$ by Proposition \ref{prop:defect-computation}.

For the converse, suppose $\mathcal{N}_P(k) \equiv \mathcal{N}_Q(\tau(k)) \mod d$ for all $k\in \Omega_b^P$ via some bijection $\tau :\Omega_b^P \rightarrow \Omega_b^Q$, and that $|\Omega^P| = |\Omega^Q|$. We see by Proposition \ref{prop:defect-computation} that $j_* \eta_{P,k}$ and $j_* \eta_{Q,\tau(k)}$ are strongly equivalent. Again by Proposition \ref{prop:defect-computation} we have that $[\iota_*\eta_{P,k}]$ and $[\iota_*\eta_{Q,\tau(k)}]$ are represented by the numbers $-|\Omega^P|$ and $-|\Omega^Q|$ which are equal by assumption. Hence, we have that $\eta_{P,k}$ and $\eta_{Q,\tau(k)}$ are strongly equivalent.
Thus, by item (1) of Proposition \ref{prop:C-env-char-ext} (with $\beta = Id$) we have that $\cenv(\tensor(P))$ and $\cenv(\tensor(Q))$ are *-isomorphic.
\end{proof}

It is interesting to try and compare these invariants with the one obtained from the graph C*-algebra of the graph of the stochastic matrix $P$. Given an irreducible graph matrix $A=(a_{ij})$ over $\Omega$, where $a_{ij}\in \{0,1\}$, in their first paper \cite{Cuntz-Krieger}, Cuntz and Krieger defined a C*-algebra $\mathcal{O}_A$ generated by partial isometries $\{S_i\}_{i\in \Omega}$ with pairwise orthogonal ranges, satisfying the relation 
$$
S_i^*S_i = \sum_{j\in \Omega} a_{ij} \cdot S_jS_j^*
$$

For a stochastic matrix $P$, one has the $\{0,1\}$-matrix $Gr(P)$ representing the directed graph of $P$. Since the C*-correspondence $Arv(P)_1$ is exactly the graph C*-correspondence of $Gr(P)$, we get that the Cuntz-Pimsner algebra $\mathcal{O}(Arv(P)_1)$ is *-isomorphic to the Cuntz-Krieger algebra $\mathcal{O}_{Gr(P)}$. In particular, by \cite[Corollary 7.4]{Kats} we see that $\mathcal{O}_{Gr(P)}$ is nuclear.

In \cite{CuntzII}, Cuntz computed the K-theory of these C*-algebras. He showed that for finite $\{0,1\}$ matrix $A$ over $\Omega$ where every column and row is non-zero, the $K_0$ and $K_1$ groups of $\mathcal{O}_A$ are given as the cokernel and kernel of the map $I-A^t : \mathbb{Z}^{\Omega} \rightarrow \mathbb{Z}^{\Omega}$.

In the case where $A$ is an irreducible finite matrix which is not a permutation matrix, Cuntz and Krieger establish in \cite{Cuntz-Krieger} that $\mathcal{O}_A$ is simple and purely infinite. Hence, for a finite irreducible stochastic matrix $P$ which is not a permutation matrix, the Cuntz-Krieger algebra $\mathcal{O}_{Gr(P)}$ is separable, unital, nuclear, simple and purely infinite, or in other words a Kirchberg algebra.

A famous classification theorem of Kirchberg and Phillips then comes into play to show that for two finite irreducible stochastic matrices $P$ and $Q$ which are not permutation matrices, the Cuntz-Krieger algebras $\mathcal{O}_{Gr(P)}$ and $\mathcal{O}_{Gr(Q)}$ are *-isomorphic ( or stably isomorphic) if and only if $(K_0(\mathcal{O}_{Gr(P)}),[1_P]_0) \cong (K_0(\mathcal{O}_{Gr(Q)}),[1_Q]_0)$ and $K_1(\mathcal{O}_{Gr(P)}) \cong K_1(\mathcal{O}_{Gr(Q)})$ ( or $K_0(\mathcal{O}_{Gr(P)}) \cong K_0(\mathcal{O}_{Gr(Q)})$ and $K_1(\mathcal{O}_{Gr(P)}) \cong K_1(\mathcal{O}_{Gr(Q)})$ respectively). That is, the *-isomorphism and stable isomorphism class are completely determined by $K$-theory. 

\begin{example} \label{ex:distinct-invariants}
In this example, we will use the above to show that for a finite irreducible stochastic matrix $P$, the Cuntz-Krieger algebra $\mathcal{O}_{Gr(P)}$ and the C*-envelope $\cenv(\tensor(P))$ generally yield incomparable invariants for $P$. If we restrict to matrices $P$ with multiple-arrival, we have that $\Omega_e = \Omega - \Omega_b$ and the invariant $\cenv(\tensor(P))$ will only depend on the graph $Gr(P)$. Hence, we will only specify the $\{0,1\}$ graph incidence matrices of three stochastic matrices $P, Q, R$. Suppose the graph matrices for $P, Q, R$ are given respectively by

\begin{equation*}
Gr(P) =  \begin{bmatrix}
       0 & 0  & 1 \\
       0  & 0  & 1 \\
       1 & 1 & 0 
     \end{bmatrix} \ , \ 
Gr(Q) =  \begin{bmatrix}
       1 & 1  & 0 \\
       1  & 1  & 1 \\
       1 & 1 & 1 
     \end{bmatrix} \ , \ 
Gr(R) =  \begin{bmatrix}
       1 & 1  & 1 \\
       1  & 1  & 1 \\
       1 & 1 & 0 
     \end{bmatrix}
\end{equation*}
then $P, Q$ and $R$ have multiple-arrival, and it is clear that $\mathcal{N}_{P}(j)= \mathcal{N}_{Q}(j) = \mathcal{N}_{R}(j) = 0$ for $j=1,2$, and that $\mathcal{N}_{P}(3) = 0$. We also see that $\mathcal{N}_{Q}(3) = \mathcal{N}_{R}(3) = 1$, so that $\cenv(\tensor(Q)) \cong \cenv(\tensor(R))$. However, $\Omega_e^{P} = \{3\}$ whereas $\Omega_e^{Q} = \Omega_e^{R} = \emptyset$, and hence $\cenv(\tensor(Q))$ is not stably isomorphic to $\cenv(\tensor(P))$.

For the Cuntz-Krieger C*-algebras the situation is reversed. The maps $I - Gr(P)^t$, $I - Gr(Q)^t$ and $I - Gr(R)^t$ on $\mathbb{Z}^3$ determining $K_0$ and $K_1$ for the Cuntz-Krieger algebras are given respectively by the matrices
\begin{equation*}
\begin{bmatrix}
       1 & 0  & -1 \\
       0  & 1  & -1 \\
       -1 & -1 & 1 
     \end{bmatrix} \ , \ 
\begin{bmatrix}
       0 & -1  & -1 \\
       -1  & 0  & -1 \\
       0 & -1 & 0 
     \end{bmatrix} \  \text{and} \
\begin{bmatrix}
       0 & -1  & -1 \\
       -1  & 0  & -1 \\
       -1 & -1 & 1 
     \end{bmatrix}
\end{equation*}
Hence, we see that the $K_1$ groups for $\mathcal{O}_{Gr(P)}$, $\mathcal{O}_{Gr(Q)}$ and $\mathcal{O}_{Gr(R)}$ are trivial, and that $\Ran(I - Gr(P)^t) = \Ran(I-Gr(Q)^t) = \mathbb{Z}^3$, so that $K_0(\mathcal{O}_{Gr(P)}) = K_0(\mathcal{O}_{Gr(Q)})$ are trivial. Hence, by the above mentioned result of Kirchberg and Phillips, we have that $\mathcal{O}_{Gr(P)}$ is *-isomorphic to $\mathcal{O}_{Gr(Q)}$. However, since $\Ran(I - Gr(R)^t) \subsetneq \mathbb{Z}^3$, we see that the cokernel $K_0(\mathcal{O}_{Gr(R)})$ is non-trivial, and hence $\mathcal{O}_{Gr(R)}$ is not stably isomorphic to $\mathcal{O}_{Gr(P)}$. Altogether, we obtain that
$$
\cenv(\tensor(P)) \not \sim \cenv(\tensor(Q)) \cong \cenv(\tensor(R)) \qquad \text{and} \qquad \cuntz_{Gr(P)} \cong \cuntz_{Gr(Q)} \not \sim \cuntz_{Gr(R)}
$$
where $\cong$ stands for *-isomorphism and $\sim$ stands for stable isomorphism. Note that the C*-envelope loses considerable information about the tensor algebra, for instance, the graphs of $P$ and $Q$ are not isomorphic so by \cite[Theorem 7.29]{dor-on-markiewicz} $\tensor(Q)$ and $\tensor(R)$ are not even algebraically isomorphic.
\end{example}

\section*{Acknowledgments} We would like to thank Orr Shalit for his many helpful remarks and suggestions on a draft version of this paper.

\end{document}